\documentclass[11pt,a4paper]{article}
\usepackage{}
\usepackage{amsfonts}
\usepackage{mathrsfs}
\usepackage{multirow}
 \usepackage{epsfig}
 \usepackage{epstopdf}
\usepackage[T1]{fontenc}
\usepackage{geometry}
\usepackage{amsbsy,amsmath,latexsym,amsfonts, epsfig, color, authblk, amssymb, graphics, bm}
\usepackage{epsf,slidesec,epic,eepic}
\usepackage{fancybox}
\usepackage{fancyhdr}
\usepackage{setspace}
\usepackage{nccmath}
\usepackage{cases}
\usepackage[colorlinks, citecolor=blue]{hyperref}

\newtheorem{theorem}{Theorem}[section]
\newtheorem{lemma}{Lemma}[section]

\newtheorem{definition}{Definition}[section]
\newtheorem{example}{Example}[section]
\newtheorem{proposition}{Proposition}[section]
\newtheorem{corollary}{Corollary}[section]

\newtheorem{remark}{Remark}[section]

\newtheorem{alemma}{Lemma}

\newenvironment{proof}{{\noindent \bf Proof:}}{\hfill$\Box$\medskip}

\definecolor{lred}{rgb}{1,0.8,0.8}
\definecolor{lblue}{rgb}{0.8,0.8,1}
\definecolor{dred}{rgb}{0.6,0,0}
\definecolor{dblue}{rgb}{0,0,0.5}
\definecolor{dgreen}{rgb}{0,0.5,0.5}

 \title{Second-order optimality conditions for SDCMPCC and application to rank optimization problems}

\author{Yulan Liu\footnote{School of Applied Mathematics, Guangdong University of Technology, Guangzhou.}\ \ {\rm and}\ \ Shaohua Pan\footnote{Corresponding author\,(shhpan@scut.edu.cn), School of Mathematics, South China University of Technology, Guangzhou.}}

 \date{}

\begin{document}

  \maketitle

  \begin{abstract}
   This paper is concerned with second-order optimality conditions for
   the mathematical program with semidefinite cone complementarity constraints (SDCMPCC).
   To achieve this goal, we first provide an exact characterization on
   the second-order tangent set to the semidefinite cone complementarity (SDCC) set
   in terms of the second-order directional derivative of the projection operator
   onto the SDCC set, and then use the second-order tangent set to the SDCC set
   to derive second-order necessary and sufficient conditions for SDCMPCC under
   suitable subregularity constraint qualifications. An application is also illustrated
   in characterizing the second-order necessary optimality condition
   for the semidefinite rank regularized problem.
  \end{abstract}

 \section{Introduction}\label{sec1}

  Let $\mathbb{X}$ and $\mathbb{Y}$ be finite dimensional vector spaces equipped with
  the inner product $\langle \cdot,\cdot\rangle$ and its induced norm $\|\cdot\|$.
  Let $\mathbb{S}^n$ denote the space of all $n\times n$ real symmetric matrices,
  equipped with the trace inner product $\langle X,Y\rangle ={\rm tr}(XY)$
  for $X,Y\in\mathbb{S}^n$ and its induced Frobenius norm $\|\cdot\|_F$,
  and let $\mathbb{S}_{+}^n$ and $\mathbb{S}_{-}^n$ denote the cone consisting
  of all semidefinite and negative semidefinite matrices of $\mathbb{S}^n$,
  respectively. Given twice differentiable functions $\varphi:\mathbb{X}\to \mathbb{R}$,
  $h\!:\mathbb{X}\to\mathbb{Y}$ and $\theta,\zeta:\mathbb{X}\to\mathbb{S}^n$,
  we are interested in the following SDCMPCC:
  \begin{equation}\label{SDCMPCC}
   \min_{x\in\mathbb{X}}\Big\{\varphi(x)\ \ {\rm s.t.}\ \
   h(x)\in K,\,\mathbb{S}^n_+\ni\theta(x)\perp \zeta(x)\in \mathbb{S}^n_-\Big\}
  \end{equation}
  where $K$ is a closed convex cone in $\mathbb{Y}$. Unless otherwise stated,
  $K$ is assumed to be second-order regular. As an extension of the mathematical
  program with the polyhedral-cone complementarity constraints \cite{Luo96}
  and the mathematical program with the second-order cone complementarity constraints
  \cite{Liang14,YeZhou16,YeZhou18,ZhangWZ15}, this problem has wide and important applications in a host of fields
  such as statistics, control and system identification, machine learning,
  combinatorial optimization, and so on since rank optimization problems and
  robust optimization problems can be reformulated as \eqref{SDCMPCC}
  (see the examples in \cite{BiPan17,DingSY14,LiuBiPan18}).

 \medskip

 Write $\Omega:=\big\{(X,Y)\in \mathbb{S}^{n}\times \mathbb{S}^{n}\,|\,
 \mathbb{S}^n_+\ni X \perp Y\in \mathbb{S}^n_-\big\}$ and
 $\Theta(x):=(\theta(x);\zeta(x))$ for $x\in\mathbb{X}$. Then,
 the problem \eqref{SDCMPCC} can be compactly written as follows
 \begin{equation}\label{ESDCMPCC}
  \min_{x\in\mathbb{X}}\Big\{\varphi(x)\ \ {\rm s.t.}\ \ (h(x);\Theta(x))\in K\times\Omega\Big\}.
 \end{equation}
 Owing to the SDC complementarity constraint $\Theta(x)\in\Omega$,
 this problem is notoriously difficult whether from an optimization theory
 or numerical algorithm standpoint. Although it can be reformulated as
 the following convex cone constrained optimization problem
 \begin{equation}\label{ConvexCone}
   \min_{x\in\mathbb{X}}\Big\{\varphi(x)\ \ {\rm s.t.}\ \
   (h(x);\Theta(x);\langle \theta(x),\zeta(x)\rangle)
   \in K\times(\mathbb{S}^n_+\times\mathbb{S}^n_-)\times\mathbb{R}_{+}\Big\},
 \end{equation}
 the common Robinson's constraint qualification (CQ) fails to hold
 at each feasible point (see \cite[Proposition 4.1]{DingSY14}). Inspired by this fact,
 Ding et al. \cite{DingSY14} studied the first-order necessary optimality
 conditions for the problem \eqref{SDCMPCC} by characterizing the limiting normal cone
 to the SDCC set $\Omega$ and introduced several kinds of important stationary points.
 Later, Wu et al. \cite{WuZZ14} provided another two classes of stationary points
 by characterizing the tangent cone to $\Omega$. Although there are some works
 on the second-order optimality condition for the mathematical program
 with the polyhedral conic complementarity constraints \cite{Gfrerer14,LuoYe13,Scheel00}, to the best of our knowledge,
 there are few works to focus on the second-order optimality conditions of \eqref{ESDCMPCC}
 except \cite{WuZZ14}, in which the authors proposed a second-order sufficient
 condition based on the equivalent conic optimization reformulation \eqref{ConvexCone}.

 \medskip

 Different from \cite{WuZZ14}, in this work we investigate the second-order
 optimality conditions starting from the problem \eqref{SDCMPCC} itself.
 Specifically, we first provide an exact characterization on the second-order
 tangent set to the SDCC set $\Omega$ by establishing its relation with
 the second-order (parabolically) directional derivative of the projection
 operator onto $\Omega$ and using the second-order directional derivative
 for symmetric matrix-valued functions \cite{ZhangZX13}; and then
 derive the second-order necessary and sufficient conditions for \eqref{SDCMPCC}
 in terms of the second-order tangent set to $\Omega$ under suitable
 metric subregularity CQs. It is known that the subregularity CQ is
 a very weak condition required for studying optimality theories of non-polyhedral
 conic optimization problems and some rules weaker than Robinson's CQ have been
 developed to identify whether the subregularity CQ holds or not
 in the past ten years (see, e.g., \cite{Gfrerer11,Gfrerer13,Henrion05,Ioffe08}).
 Although a no gap second-order necessary and sufficient optimality
 condition is also obtained under a stronger assumption, the gap-type second-order
 necessary optimality conditions are still weaker than the existing ones.

 \medskip

 It is worthwhile to point out that the second-order tangent sets to
 closed convex sets and convex conic constrained systems were well studied
 in the past decade (see, e.g., \cite{BonnansCS99,BonnansR05,BS00,Constantin06}),
 but there are few works for the second-order tangent set to
 a non-polyhedral conic complementarity constraint system. Recently,
 Chen and Ye \cite{ChenYe19} characterized the second-order tangent set
 to the second-order cone complementarity constraint system, but only
 derived a second-order necessary optimality condition under a constraint
 nondegeneracy assumption. This paper is partly motivated by their work.
 We not only characterize the second-order tangent set to the SDCC set,
 but also investigate the second-order necessary and sufficient optimality
 conditions under suitable metric subregularity CQs. Just when this work is finished,
 we learned that Gfrerer et al. \cite{Gfrerer19} under weaker conditions established
 second-order optimality conditions for nonconvex set-constrained optimization problems,
 which covers the problem \eqref{ESDCMPCC} as a special case. However, the second-order
 necessary optimality condition obtained there is weaker than ours (see Corollary \ref{SONC2}),
 and moreover, it seems much more difficult to characterize the lower generalized support function
 of the second-order tangent set to $\Omega$.
  \section{Notation and preliminaries}\label{sec1}

  In this paper, we denote $I,E$ and $e$ by an identity matrix,
  a matrix of all ones and a vector of all ones, respectively,
  whose dimensions are known from the context. For a vector $z$,
  ${\rm Diag}(z)$ means the diagonal matrix with the $i$-th diagonal
  entry being $z_i$; and for given $Z_i\in\mathbb{R}^{m_i\times m_i}$
  for $i=1,\ldots,p$, ${\rm Diag}(Z_1,\cdots, Z_p)$ means the block diagonal
  matrix with the $i$-th diagonal block being $Z_i$. The notation
  $\mathbb{O}^{m\times n}$ denotes the set of $m\times n$ matrices with
  orthonormal columns and $\mathbb{O}^{m\times m}$ means $\mathbb{O}^{m}$.
  For a given $Z\in\mathbb{S}^n$, write $\Lambda(Z):={\rm Diag}(\lambda(Z))$
  and $\mathbb{O}^{n}(Z):=\{P\in\mathbb{O}^{n}\,|\, Z=P\Lambda(Z)P^{\mathbb{T}}\}$,
  where $\lambda(Z)\in\mathbb{R}^n$ means the eigenvalue vector arranged in
  a nonincreasing order; for given index sets $\alpha\subseteq\{1,\ldots,n\}$
  and $\beta\subseteq\{1,\ldots,n\}$, $Z_{\alpha\beta}$ means the submatrix consists
  of those entries $Z_{ij}$ with $i\in\alpha$ and $j\in\beta$, and $Z_{\alpha}$
  means the matrix consisting of those columns $Z_j$ with $j\in\alpha$.
  For a closed set $S\subseteq\mathbb{X}$, $\delta(\cdot\,|\,S)$ and
  $\sigma(\cdot\,|\,S)$ denote the indicator function and the support function
  of the set $S$, respectively. For a closed cone $C$, $C^{\circ}$ means
  the negative polar of $C$, and if $C$ is convex, ${\rm lin}(C)$ means the largest
  subspace contained in $C$.

  \medskip

  For any given $t_1,t_2\in\mathbb{R}$, if a scalar function
  $g\!:\mathbb{R}\to\mathbb{R}$ is differentiable at $t_1$ and $t_2$,
  the notation $g^{[1]}(t_1,t_2)$ represents the first divided difference
  of $g$ at $(t_1,t_2)$, defined by
  \[
   g^{[1]}(t_1,t_2):=\left\{\begin{array}{cl}
                   \frac{g(t_1)-g(t_2)}{t_1-t_2} & {\rm if}\ t_1\ne t_2, \\
                          g'(t_1) & {\rm otherwise}.
                   \end{array}\right.
  \]
  If $g$ is differentiable at each component of a vector
  $\lambda=(\lambda_1,\ldots,\lambda_n)^{\mathbb{T}}$,
  $g^{[1]}({\rm Diag}(\lambda))$ denotes an $n\times n$ symmetric matrix
  whose $(i,j)$-th entry is $g^{[1]}(\lambda_i,\lambda_j)$.
  For any given $t_1,t_2,t_3\in\mathbb{R}$,
  if the scalar function $g$ is twice differentiable at each $t_i$,
  $g^{[2]}(t_1,t_2,t_3)$ denotes the second divided difference
  of $g$ at $(t_1,t_2,t_3)$ defined as follows: if $t_1,t_2,t_3$ are
  distinct,
  \[
    g^{[2]}(t_1,t_2,t_3):=\frac{g^{[1]}(t_1,t_2)-g^{[1]}(t_1,t_3)}{t_2-t_3};
  \]
  and for other values of $t_1,t_2,t_3$, $g^{[2]}(t_1,t_3,t_3)$ is defined by continuity, e.g.,
  \[
   g^{[2]}(t_1,t_3,t_3):=\frac{g^{[1]}(t_1,t_3)-g'(t_3)}{t_1-t_3}.
  \]
  For a function $f\!:\mathbb{R}\to\mathbb{R}$,
  the associated L\"{o}wner operator $\mathscr{F}\!:\mathbb{S}^n\to\mathbb{S}^n$
  is defined as
  \[
    \mathscr{F}(X)=P{\rm Diag}(f(\lambda_1(X)),\ldots,f(\lambda_n(X)))P^{\mathbb{T}}
    \quad{\rm for}\ X=P{\rm Diag}(\lambda(X))P^{\mathbb{T}}.
  \]

  First of all, we recall from \cite{BS00,RW98} three classes of tangent cones
  to a closed set $S\subseteq\mathbb{X}$ and the outer and inner
  second-order tangent sets to the closed set $S$.
  \begin{definition}\label{Def-SOTS}
   Let $S\subset\mathbb{X}$ be a closed set. Consider an $\overline{x}\in S$
   and a direction $d\in\mathbb{X}$. The tangent cone, inner tangent cone
   and Clarke tangent cone to $S$ at $\overline{x}$ are defined as
   \begin{subequations}
    \begin{align*}
    \mathcal{T}_{S}(\overline{x}):=\big\{d\in\mathbb{X}\ |\ \exists\, t_k\downarrow 0,\,
                   d^k\to d\ {\rm with}\ \overline{x}+t_kd^k\in S\big\},\qquad\\
     \mathcal{T}^{i}_{S}(\overline{x}):=\big\{d\in\mathbb{X}\ |\ \forall\, t_k\downarrow 0,\, \exists \,
      d^k\to d\ {\rm with}\ \overline{x}+t_kd^k\in S\big\},\qquad\\
    \mathcal{T}^{c}_{S}(\overline{x}):=\big\{d\in\mathbb{X}\ |\ \forall\, t_k\downarrow 0,\, x^k\xrightarrow[S]{}\overline{x},\ \exists \, d^k\to d\ {\rm with}\ x_k+t_kd^k\in S\big\}.
    \end{align*}
   \end{subequations}
   The outer and inner second-order tangent sets to $S$ at $\overline{x}$ in direction $d$
   are defined as
   \begin{subequations}
    \begin{align*}
    \mathcal{T}^{2}_{S}(\overline{x};d)
     :=\big\{w\in\mathbb{X}\ |\ \exists\, t_k\downarrow 0\ {\rm such\ that}\ {\rm dist}\big(\overline{x}+t_kd+\frac{1}{2}t_k^2 w,S\big)=o(t_k^2)\big\},\\
    \mathcal{T}^{i,2}_{S}(\overline{x};d):=\big\{w\in\mathbb{X}\ |\ {\rm dist}\big(\overline{x}+td+\frac{1}{2}t^2 w,S\big)=o(t^2)\ {\rm for\ any}\ t\ge 0\big\}.\quad
   \end{align*}
   \end{subequations}
  \end{definition}

  The tangent cones $\mathcal{T}^{c}_{S}(\overline{x}),\mathcal{T}^{i}_{S}(\overline{x})$
  and $\mathcal{T}_{S}(\overline{x})$ are all closed and satisfy
  $\mathcal{T}^{c}_{S}(\overline{x})\subset\mathcal{T}^{i}_{S}(\overline{x})
  \subset\mathcal{T}_{S}(\overline{x})$, and $\mathcal{T}^{c}_{S}(\overline{x})$ is convex.
  It is clear that $\mathcal{T}^{i,2}_{S}(\overline{x};h)
  \subset \mathcal{T}^{2}_{S}(\overline{x};h)$. Notice that
  $\mathcal{T}^{i,2}_{S}(\overline{x};h)$ and $\mathcal{T}^{2}_{S}(\overline{x};h)$
  are generally not a cone, and they can be nonempty only if $h\in \mathcal{T}_{S}(\overline{x})$.
  \begin{definition}\label{NormDef}
   (see \cite{Mordu06,RW98}) Let $S\subset\mathbb{X}$ be a given set.
   Consider an arbitrary $\overline{x}\in S$ and a direction $d\in\mathbb{X}$.
   The regular/Fr\'{e}chet normal cone to $S$ at $x$ is defined by
   \[
     \widehat{\mathcal{N}}_{S}(\overline{x})
     :=\big\{v\in\mathbb{X}\ |\ \langle v,x'-\overline{x}\rangle\le o(\|x'-\overline{x}\|)\quad\forall x'\in S\big\};
    \]
   the limiting/Mordukhovich normal cone to $S$ at $\overline{x}$ is defined as
   \[
     \mathcal{N}_{S}(\overline{x}):=\limsup_{x'\xrightarrow[S]{}\overline{x}}\widehat{\mathcal{N}}_{S}(x')
     =\Big\{v\in \mathbb{X}\,|\, \exists\, x^k\xrightarrow[S]{}\overline{x},v^k\to v {\ \ \rm with \ \ } v^k\in \widehat{\mathcal{N}}_{S}(x^k)\Big\};
    \]
   and the Clark normal cone to $S$ at $\overline{x}$ is defined as
   $\mathcal{N}^c_{S}(\overline{x}):={\rm cl co}\mathcal{N}_{S}(\overline{x})$;
   while the limiting normal cone to the set $S$ in the direction $d$ at $\overline{x}$ is defined by
   \[
     \mathcal{N}_{S}(\overline{x};d)\!:=\limsup_{t\downarrow 0,d'\to d}\widehat{\mathcal{N}}_{S}(x+td')
     =\!\Big\{v\in \mathbb{X}\ |\ \exists t_k\downarrow 0,d^k\to d,v^k\to v\  {\rm with}\
     v^k\!\in\!\widehat{\mathcal{N}}_{S}(x+t_kd^k)\Big\}.
    \]
  \end{definition}

  The directional limiting normal cone was introduced in \cite{Ginchev11}.
  By Definition \ref{NormDef}, clearly, $\mathcal{N}_S(\overline{x};d)=\emptyset$
  if $d\notin\mathcal{T}_S(x)$, $\mathcal{N}_S(\overline{x};d)\subseteq\mathcal{N}_S(\overline{x})$
  and $\mathcal{N}_S(\overline{x};0)=\mathcal{N}_S(\overline{x})$.
  By \cite[Theorem 6.28]{RW98}, the above tangent cones and normal cones
  have the following polar relations
  \[
    \widehat{\mathcal{N}}_S(\overline{x})=[\mathcal{T}_{S}(\overline{x})]^{\circ},\,
    \mathcal{T}^{c}_{S}(\overline{x})=[\mathcal{N}_S(\overline{x})]^{\circ}
    \ {\rm and}\
    \mathcal{N}_S^c(\overline{x})={\rm cl}({\rm co}\mathcal{N}_S(\overline{x}))
    =[\mathcal{T}^{c}_{S}(\overline{x})]^{\circ}.
  \]

  The following definition recalls the metric subregularity of a multifunction
  in a certain direction introduced in \cite{Gfrerer13}, which is weaker than
  its metric subregularity.
 \begin{definition}
  (see \cite[Definition 1]{Gfrerer13}) Let $\Upsilon\!:\mathbb{X}\to\mathbb{Y}$ be
  a mapping and $S\subseteq\mathbb{Y}$ be a closed set. We say that
  the multifunction $\mathcal{F}(z)\!:=\Upsilon(z)-S$ is subregular at
  $\overline{x}$ for the origin in direction $d\in\mathbb{X}$ if there exist
  $\kappa>0,\rho>0$ and $\delta>0$ such that for all $x\in\overline{x}+V_{\rho,\delta}$,
  \[
    {\rm dist}(x,\mathcal{F}^{-1}(0))\le\kappa{\rm dist}(\Upsilon(x),S)
  \]
  where $V_{\rho,\delta}:=\{w\in\rho\mathbb{B}\ |\ \|\|d\|w-\|w\|d\|\le\delta\|w\|\|d\|\}$
  is a directional neighborhood of the direction $d$. When $d=0$, we say that
  $\mathcal{F}$ is metrically subregular at $\overline{x}$ for the origin.
 \end{definition}

  Let $\Upsilon\!:\mathbb{X}\to\mathbb{Y}$ be a twice differentiable mapping and
  $S\subseteq\mathbb{Y}$ be a closed set. We next pay our attention to the set
  $\Sigma\subseteq\mathbb{X}$ that can be represented locally at
  a fixed $\overline{x}\in\mathbb{X}$ as
  \begin{equation}\label{nconvex-system}
   \Sigma\cap\mathcal{O}=\big\{x\in\mathcal{O}\ |\ \Upsilon(x)\in S\big\}
  \end{equation}
  where $\mathcal{O}$ is a neighborhood of $\overline{x}$.
  For such a system, the following results hold.
 \begin{lemma}\label{MSCQ-result}
  Let $\Sigma$ be the constrained system represented by \eqref{nconvex-system}
  around $\overline{x}\in\Sigma$. Suppose the mapping
  $\mathcal{F}(z)\!:=\Upsilon(z)-S$ is subregular at $\overline{x}$
  for the origin in direction $d\in\mathbb{X}$. Then
  \begin{itemize}
   \item[(i)] $d\in\mathcal{T}_{\Sigma}(\overline{x})$ if and only if $
                    \Upsilon'(\overline{x})d\in\mathcal{T}_{S}(\Upsilon(\overline{x}))$;

   \item[(ii)] if $d\in\mathcal{T}_{\Sigma}(\overline{x})$, then
               $w\in\mathcal{T}_{\Sigma}^2(\overline{x},d)$ iff
                $\Upsilon'(\overline{x})w+\nabla^2\Upsilon(\overline{x})(d,d)
                \in\mathcal{T}_{S}^2(\Upsilon(\overline{x}),\Upsilon'(\overline{x})d)$;

   \item[(iii)] if $\Upsilon'(\overline{x})d \in\mathcal{T}_{S}(\Upsilon(\overline{x}))$,
                then there exists $\kappa>0$ such that
                \(
                  \mathcal{S}_d(p)\subseteq\mathcal{S}_d(0)+\kappa\|p\|\mathbb{B}_{\mathbb{X}},
                \)
                where $\mathcal{S}_{\xi}(p):=\{u\in\mathbb{X}\ |\ \Upsilon'(\overline{x})u
                +\nabla^2\Upsilon(\overline{x})(\xi,\xi)+p             \in\mathcal{T}_{S}^2(\Upsilon(\overline{x}),\Upsilon'(\overline{x})\xi)\big\}$ for $p\in\mathbb{Y}$.
  \end{itemize}
  \end{lemma}

  The proof of part (i) and (ii) can be found in \cite[Proposition 4.1]{Gfrerer162}
  and \cite[Proposition 5]{Gfrerer19}, and part (iii) can be achieved by
  using the same arguments as those in \cite[Theorem 4.3]{Mohammadi19}.
  Notice that part (i) and part (ii)-(iii) were established
  under the subregularity of $\mathcal{F}$ at $(\overline{x},0)$
  in \cite[Proposition 1]{Henrion05}
  and \cite[Theorem 4.3 \& Theorem 4.5]{Mohammadi19}, respectively.

  \medskip

  By invoking Lemma \ref{MSCQ-result}, we have the following conclusion,
  which generalizes the result of \cite[Proposition 3.88]{BS00} to
  the nonconvex system \eqref{nconvex-system} under the metric subregularity.
  \begin{proposition}\label{SO-regular}
   Let $S\subseteq\mathbb{Y}$ be a closed set and $\Upsilon\!:\mathbb{X}\to\mathbb{Y}$
   be a twice differentiable mapping. Consider an arbitrary $\overline{x}\in\Upsilon^{-1}(S)$
   and a direction $d\in\mathbb{X}$. Suppose that $\mathcal{F}(z)\!:=\Upsilon(z)-S$
   is subregular at $\overline{x}$ for the origin in the direction $d$ and $S$ is outer
   second-order regular (second-order regular) at $\Upsilon(\overline{x})$ in the direction
   $\Upsilon'(\overline{x})d\in\mathcal{T}_{S}(\Upsilon(\overline{x}))$.
   Then $\Upsilon^{-1}(S)$ is outer second-order regular (second-order regular)
   at $\overline{x}$ in the direction $d$.
  \end{proposition}
  \begin{proof}
   By Lemma \ref{MSCQ-result}(i), $d\in\mathcal{T}_{\Upsilon^{-1}(S)}(\overline{x})$.
   Let $x^k\!:=\overline{x}+t_kd+\frac{1}{2}t_k^2w^k\in \Upsilon^{-1}(S)$
   be an arbitrary sequence with $t_k\downarrow 0$ and $t_kw^k\to 0$.
   A second-order expansion of $\Upsilon$ at $x^k$ is
   \[
     S\ni \Upsilon(x^k)=\Upsilon(\overline{x})+t_k\Upsilon'(\overline{x})d
     +\frac{1}{2}t_k^2\big[\Upsilon'(\overline{x})w^k+\nabla^2\Upsilon(\overline{x})(d,d)\big]+o(t_k^2).
   \]
   Since $S$ is outer second-order regular at $\Upsilon(\overline{x})$
   in $\Upsilon'(\overline{x})d$, there exists $\eta_k\to 0$ such that
   \begin{equation}\label{ineq1-prop88}
     \Upsilon'(\overline{x})w^k+\nabla^2\Upsilon(\overline{x})(d,d)+\eta_k\in \mathcal{T}_S^2(\Upsilon(\overline{x}),\Upsilon'(\overline{x})d).
   \end{equation}
   Let $\mathcal{S}_d\!:\mathbb{Y}\to\mathbb{X}$ be the multifunction
   defined as in Lemma \ref{MSCQ-result}(iii). By Lemma \ref{MSCQ-result}(iii),
   there exists $c>0$ such that for any $p\in\mathbb{R}^n$,
   \[
     \mathcal{S}_d(p)\subset\mathcal{S}_d(0)+c\|p\|\mathbb{B}_{\mathbb{X}}.
   \]
   From equation \eqref{ineq1-prop88} we know that $w^k\in\mathcal{S}_d(\eta_k)$ for each $k$.
   Along with the last inclusion, there exists $\widetilde{w}^k\in\mathcal{S}_d(0)$ such that
   $\|w^k-\widetilde{w}^k\|\le c\eta_k$ for each $k$. That is, for each $k$,
   there exists $\xi^k$ with $\|\xi^k\|\le c\eta_k$ such that
   $\widetilde{w}^k=w^k+\xi^k$. Together with $\widetilde{w}^k\in\mathcal{S}_d(0)$, we have
   \[
     \Upsilon'(\overline{x})(w^k+\xi^k)+\nabla^2\Upsilon(\overline{x})(d,d)
     \in\mathcal{T}_S^2(\Upsilon(\overline{x}),\Upsilon'(\overline{x})d).
   \]
   This, by Lemma \ref{MSCQ-result}(ii), implies that
   $w^k+\xi^k\in\mathcal{T}_{\Upsilon^{-1}(S)}^{2}(\overline{x},d)$, and consequently,
   \[
     {\rm dist}(w^k,\mathcal{T}_{\Upsilon^{-1}(S)}^{2}(\overline{x},d))
     \le \|\xi^k\|\to 0.
   \]
   Hence, $\Upsilon^{-1}(S)$ is outer second-order regular at $\overline{x}$
   in the direction $d$. If $S$ is second-order regular at $\Upsilon'(\overline{x})$
   in the direction $\Upsilon'(\overline{x})d$, the above arguments show that
   $\Upsilon^{-1}(S)$ is outer second-order regular at $\overline{x}$ in the direction $d$.
   The proof is completed.
  \end{proof}

  \section{Second-order directional derivative of $\Pi_{\mathbb{S}^n_+}$}\label{sec3}

  Let $\Pi_{\mathbb{S}^n_+}(\cdot)$ denote the projection operator onto $\mathbb{S}_{+}^n$.
  Notice that $\Pi_{\mathbb{S}^n_+}(\cdot)$ is the L\"{o}wner operator associated
  to the function $t\mapsto\max(0,t)$. In this part, we use \cite[Theorem 4.1]{ZhangZX13}
  to obtain the formula for calculating the second-order directional derivative
  of $\Pi_{\mathbb{S}^n_+}(\cdot)$. In order to state \cite[Theorem 4.1]{ZhangZX13},
  we need to introduce some necessary notation.

  \medskip

  For any given $Z\in\mathbb{S}^n$, let the matrix $Z$ have
  the eigenvalue decomposition as follows
  \begin{subnumcases}{}
   \label{Zdecom}
   Z=\big[P_{\alpha}\ \ P_{\beta}\ \ P_{\gamma}\big]
      \left[\begin{matrix}
           {\rm Diag}(\lambda_{\alpha}(Z)) & 0 & 0\\
            0 & 0 & 0\\
            0 & 0 &  {\rm Diag}(\lambda_{\gamma}(Z))
      \end{matrix}\right]\big[P_{\alpha}\ \ P_{\beta}\ \ P_{\gamma}\big]^{\mathbb{T}},\\
   \label{idx-eigZ}
   \alpha:=\big\{i\,|\, \lambda_i(Z)>0\big\},\
   \beta:=\big\{i\,|\, \lambda_i(Z)=0\big\}\ \ {\rm and}\ \
   \gamma:=\big\{i\,|\, \lambda_i(Z)<0\big\}.
   \end{subnumcases}
  Let $\mu_1>\mu_2>\cdots>\mu_r$ be the distinct eigenvalues of $Z$ and
  assume that $\mu_{r_0}=0$ for some $r_0\in\{1,2,\ldots,r\}$.
  For each $k\in\{1,2,\ldots,r\}$, define the index set
  \begin{equation}\label{akDef}
   a_{k}:=\big\{i\in\{1,2,\ldots,n\}\ |\ \lambda_i(Z)=\mu_k\big\}.
  \end{equation}
  Clearly, $\alpha=\bigcup_{i=1}^{r_0-1}a_i,\beta=a_{r_0}$ and $\gamma=\bigcup_{i=r_0+1}^{r}a_i$.
  The matrix $P$ can be partitioned as
  \[
    P=\big[P_{a_1}\ P_{a_2}\ \cdots\ P_{a_r}\big]\ \ {\rm with}\ P_{a_k}\in \mathbb{O}^{n\times |a_k|}
    \ \ {\rm for}\ \ k=1,\ldots,r.
  \]
  Let $f\!:\mathbb{R}\to\mathbb{R}$ be a function that is second-order
  directionally differentiable at each $\mu_k$.
  Fix an arbitrary $k\in\{1,\ldots,r\}$. There exists $\delta_k>0$ such that
  for all $l\neq k$, $|\mu_l-\mu_k|>\delta_k$. With such $\delta_k$,
  we define the continuous function $g_k\!:\mathbb{R}\to\mathbb{R}$ and
  $g\!:\mathbb{R}\to\mathbb{R}$ by
  \[
    g_k(t):=\left\{\begin{array}{cl}
    \!-\frac{6}{\delta_k}(t-\mu_k-\frac{\delta_k}{2})&{\rm if}\ t\in (\mu_k+\frac{\delta_k}{3},\mu_k+\frac{\delta_k}{2}],\\
                                1 & {\rm if}\ t\in [\mu_k-\frac{\delta_k}{3},\mu_k+\frac{\delta_k}{3}], \\
   \!\frac{6}{\delta_k}(t-\mu_k+\frac{\delta_k}{2})&{\rm if}\ t\in [\mu_k-\frac{\delta_k}{2},\mu_k-\frac{\delta_k}{3}),\\
   0 & {\rm otherwise}
   \end{array}\right.{\rm and}\ \
   g(t):=\sum_{k=1}^{r} g_k(t)f(\mu_k).
  \]
  For each $k,l\in\{1,2,\ldots,r\}$, let $G^{[2]}_{kl}\!:\mathbb{S}^n\to\mathbb{S}^n$
  be the L\"{o}wner operator associated to the function $g^{[2]}(\mu_k,\cdot,\mu_l)$.
  Obviously, $G^{[2]}_{kl}(Z)=\sum_{j=1}^{r} g^{[2]}(\mu_k,\mu_j,\mu_l)P_{a_j}P_{a_j}^{\mathbb{T}}$.

  \medskip

  For each $H\in\!\mathbb{S}^n$, we write $\widetilde{H}:=P^{\mathbb{T}}HP$.
  Fix an arbitrary $k\in\{1,2,\ldots,r\}$. Let the matrix $\widetilde{H}_{a_ka_k}\in\mathbb{S}^{|a_k|}$
  have the following eigenvalue decomposition:
  \begin{subnumcases}{}\label{Hkdecom}
  \!\qquad\widetilde{H}_{a_ka_k}=\big[Q_{\alpha^k}^k\ \ Q_{\beta^k}^k\ \ Q_{\gamma^k}^k\big]
      {\rm Diag}(\lambda(\widetilde{H}_{a_ka_k}))\big[Q_{\alpha^k}^k\ \ Q_{\beta^k}^k\ \ Q_{\gamma^k}^k\big]^{\mathbb{T}},
      \qquad\qquad\quad\\
   \label{alpk-gamk}
   \!\alpha^k\!:=\!\{j\ |\ \lambda_j(\widetilde{H}_{a_ka_k})>0\},
   \beta^k\!:=\!\{j\ |\ \lambda_j(\widetilde{H}_{a_ka_k})=0\},
   \gamma^k\!:=\!\{j\ |\ \lambda_j(\widetilde{H}_{a_ka_k})<0\}.\qquad
  \end{subnumcases}
  Let $\eta_1^k>\eta_2^k>\cdots>\eta^k_{N_k}$ be the distinct eigenvalues
  of $\widetilde{H}_{a_ka_k}$ and assume that $\eta_{N_k^0}=0$ for some
  $N_k^0\in\{1,2,\ldots,N_k\}$. For each $j=1,2,\ldots,N_k$, define the index set
  \begin{equation}
   b^k_j\!:=\!\big\{i\in\!\{1,2,\ldots,|a_k|\}\ |\ \lambda_i(\widetilde{H}_{a_ka_k})=\eta_j^k\big\}.
  \end{equation}
  Clearly, $\alpha^k=\bigcup_{j=1}^{N_k^0-1}b^k_j,\,\beta^k=b^k_{N_k^0}$
  and $\gamma^k=\bigcup_{j=N_k^0+1}^{N_k}b^k_j$. The $Q^k$ can be partitioned as
  \[
    Q^k=\big[Q_{b_1^k}^k\ \ Q_{b_2^k}^k\ \ \cdots\ \ Q_{b_{N_k}^k}^k\big]
    \ \ {\rm with}\ Q_{b_j^k}^k\in \mathbb{O}^{|a_k|\times |b_j^k|}
    \ \ {\rm for}\ \ j=1,\ldots,N_k.
  \]
  Fix an arbitrary $j\in\{1,\ldots,N_{k}\}$. Similarly, there exists $\delta_{j}^k>0$
  such that for all $l\ne j$, $|\eta^k_j-\eta^k_l|>\delta^k_{j}$.
  With such $\delta_j^k$, we define the following continuous functions
   \[
   \phi^k_{j}(t)\!:=\left\{\begin{array}{cl}
   \!-\frac{6}{\delta^k_{j}}(t\!-\!\eta^k_j\!-\!\frac{\delta^k_j}{2})
   & {\rm if\;} t\in (\eta^k_j\!+\!\frac{\delta^k_j}{3},\eta^k_j\!+\!\frac{\delta^k_j}{2}],\\
                                1
   & {\rm if\;} t\in [\eta^k_j\!-\!\frac{\delta^k_j}{3},\eta^k_j\!+\!\frac{\delta^k_j}{3}], \\
   \!\frac{6}{\delta^k_{j}}(t\!-\!\eta^k_j\!+\!\frac{\delta^k_j}{2})
   &{\rm if\;} t\in [\eta^k_j\!-\!\frac{\delta^k_j}{2},\eta^k_j\!-\!\frac{\delta^k_j}{3}),\\
   0 &  {\rm otherwise}
   \end{array}\right.{\rm and}\
   \phi_k(t)\!:=\!\sum_{j=1}^{N_k} f'(\mu_k;\eta^k_j)\phi^k_{j}(t).
  \]
  Clearly, for each $j\in\{1,2,\ldots,N_k\}$, $\phi_k(\eta^k_j)=f'(\mu_k;\eta^k_j)$.
  Then, it is immediate to have
  \begin{equation*}
  \!\sum_{j=1}^{N_k} f'(\mu_k;\eta^k_j)Q^k_{b^k_j}(Q^k_{b^k_j})^\mathbb{T}
  \!=Q^k\!\left[\!\begin{matrix}\phi_k(\eta^k_1)I_{|b^k_1|} & \\
            &\ddots & \\ & & \phi_k(\eta^k_{N_k})I_{|b^k_{N_k}|}\end{matrix}\!\right]\!
            (Q^k)^\mathbb{T}\!:=\Phi_{k}(\widetilde{H}_{a_ka_k}).
 \end{equation*}
 \vspace{-0.3cm}
 \begin{lemma}\label{ReviseTheo}(see \cite[Theorem 4.1]{ZhangZX13})
  Fix an arbitrary $Z\in\mathbb{S}^n$ with the spectral decomposition as
  in \eqref{Zdecom}-\eqref{idx-eigZ}. Then, $f\!:\mathbb{R}\to\mathbb{R}$
  is second-order directionally differentiable at each $\lambda_i(Z)$ if
  and only if the associated L\"{o}wner operator $\mathscr{F}$ is second-order
  directionally differentiable at $Z$. In particular, for any given
  $(H,W)\in \mathbb{S}^n\times\mathbb{S}^n$, $\mathscr{F}''(Z;H,W)=PBP^\mathbb{T}$.
  Among others, for each $k\in\{1,2,\ldots,r\}$,
  the submatrix $B_{a_ka_k}$ takes the form of
  \begin{align}\label{BKKFormula}
  &Q^k\big[\phi_{k}^{[1]}\big(\Lambda(\widetilde{H}_{a_k a_k})\big)
  \circ (Q^k)^{\mathbb{T}}\widetilde{V}_{k}(H,W)Q^k\big](Q^k)^{\mathbb{T}}
  +2I_{a_k}^{\mathbb{T}}\widetilde{H}G^{[2]}_{kk}\big(\Lambda(Z)\big)\widetilde{H}I_{a_k}\nonumber\\
  &+ Q^k{\rm Diag}\Big(\Psi_{k,1}((Q^k_{b^k_1})^{\mathbb{T}}\widetilde{V}_{k}(H,W)Q^k_{b^k_1}),\ldots,
  \Psi_{k,N_{k}}((Q^k_{b^k_{N_k}})^{\mathbb{T}}\widetilde{V}_{k}(H,W)Q^k_{b^k_{N_k}})\Big)
  (Q^k)^{\mathbb{T}}
  \end{align}
  where $\widetilde{V}_k(H,W)\!:=P^{\mathbb{T}}_{a_k}\big(W\!-2H(Z-\mu_k I)^{\dagger}H\big)P_{a_k}$
  and $\Psi_{k,j}(\cdot)$ for $j=1,2,\ldots,N_k$ is the L\"{o}wner operator associated to
  $f''(\mu_k;\eta^k_j,\cdot)$; and for each $k,l\in\{1,\ldots,r\}$ with $k\neq l$,
 \begin{subequations}
 \begin{align}\label{BKLFormula1}
  B_{a_k a_l}=\big[g^{[1]}(\Lambda(Z))\big]_{a_k a_l}\circ \widetilde{W}_{a_k a_l}
  +2I^{\mathbb{T}}_{a_k} \widetilde{H} G^{[2]}_{kl}\big(\Lambda(Z)\big)\widetilde{H} I_{a_l}\\
  \label{BKLFormula2}
  +\frac{2\widetilde{H}_{a_l a_k}\Phi_{l}(\widetilde{H}_{a_l a_l})}{\mu_l-\mu_k}
         +\frac{2\Phi_{k}(\widetilde{H}_{a_k a_k})\widetilde{H}_{a_k a_l}}{\mu_k-\mu_l}.
  \qquad\qquad
 \end{align}
 \end{subequations}
 \end{lemma}
 \begin{remark}
  Notice that the term $2I_{a_k}^{\mathbb{T}}\widetilde{H}G^{[2]}_{kk}(\Lambda(Z))\widetilde{H}I_{a_k}$
  in \eqref{BKKFormula} is lost in \cite[Theorem 4.1]{ZhangZX13}, and the term
  $2I^{\mathbb{T}}_{a_k} \widetilde{H} G^{[2]}_{kl}(\Lambda(Z))\widetilde{H} I_{a_l}$
  in \eqref{BKLFormula1} is a little different from the one there.
 \end{remark}

  Next we figure out the structure of each block $B_{a_ka_l}$ associated to
  $f(t)\!=\max(0,t)$.
 \begin{proposition}\label{BKL}
  Let $f(t)\!=\max(0,t)$ for $t\in\mathbb{R}$. Then, for each $k,l\in\{1,2,\ldots,r\}$
  with $k<l$, the submatrix $B_{a_ka_l}$ defined in \eqref{BKLFormula1}-\eqref{BKLFormula2} takes
  the following form
  \begin{subnumcases}{}
   \label{klr01}
   \widetilde{W}_{\!a_ka_l}\!+2P_{\!a_k}^{\mathbb{T}}H\Big[\!\sum_{j=r_0+1}^{r} \frac{\mu_j}{(\mu_j\!-\!\mu_k)(\mu_j\!-\!\mu_l)}P_{\!a_j}P_{\!a_j}^\mathbb{T}\Big]HP_{a_l}
   \quad{\rm if}\ k<l<r_0;\\
   \label{klr02}
   \widetilde{W}_{a_ka_l}\!+2P^{\mathbb{T}}_{a_k}H\!\sum_{j=r_0+1}^{r} \frac{P_{\!a_j}P_{\!a_j}^\mathbb{T}}{\mu_j\!-\!\mu_k}HP_{a_l}
   +\frac{2}{\mu_k}\widetilde{H}_{a_k a_l}\Pi_{\mathbb{S}^{|a_l|}_-}(\widetilde{H}_{a_l a_l})
   \ \ {\rm if}\ k<l\!=r_0;\\
   \label{klr03}
   P^{\mathbb{T}}_{a_k}H\sum_{j=1}^{r_0-1}\frac{1}{(\mu_j\!-\!\mu_l)}P_{a_j}P_{a_j}^\mathbb{T}HP_{a_l}
   -\frac{2}{\mu_l}\Pi_{\mathbb{S}^{|a_k|}_+}(\widetilde{H}_{a_k a_k})\widetilde{H}_{a_k a_l}
   \quad{\rm if}\ k=r_0<l;\quad\\
   \label{klr04}
   P^{\mathbb{T}}_{\!a_k}
   H\!\sum_{j=1}^{r_0-1}\frac{\mu_j}{(\mu_j-\mu_k)(\mu_j-\mu_l)}P_{\!a_j}P_{\!a_j}^\mathbb{T}HP_{a_l}
    \quad{\rm if}\ \ r_0<k<l;\qquad\qquad\\
   \label{klr05}
   \Sigma_{\alpha\gamma}\circ\widetilde{W}_{a_k a_l}+2 I^\mathbb{T}_{a_k} \widetilde{H}G^{[2]}_{kl}(\Lambda(Z))\widetilde{H}I_{a_l}+
   \frac{2\widetilde{H}_{a_k a_k}\widetilde{H}_{a_k a_l}}{\mu_k-\mu_l}\quad{\rm if}\ \ k<r_0<l\\
   {\rm with}\
   I^{\mathbb{T}}_{a_k}\widetilde{H}G^{[2]}_{kl}(\Lambda(Z))\widetilde{H}I_{a_l}
   =P^{\mathbb{T}}_{a_k}H\Big[\sum_{j=1}^{r} g^{[2]}(\mu_k,\mu_j,\mu_l)P_{\!a_j}P_{\!a_j}^\mathbb{T}\Big]HP_{a_l}.\nonumber
 \end{subnumcases}
 \end{proposition}
 \begin{proof}
  We proceed the arguments by three steps as will be shown below.

  \noindent
  {\bf Step 1:} to characterize $\Phi_{k}(\widetilde{H}_{a_k a_k})$.
  From the definition of $\phi_k$, it follows that
  \begin{align}\label{phiketa}
  \phi_k(\eta^k_j)=f'(\mu_k;\eta^k_j)=\left\{\begin{array}{cl}
                      \eta^k_j & {\rm if\;} \mu_k\ge0,\\
                       f(\eta^k_j) & {\rm if\;} \mu_k=0, \\
                       0 &{\rm if\;} \mu_k<0
                    \end{array}\right.\ \ {\rm for}\ i=1,2,\ldots,N_k.
  \end{align}
  Recall that $\Phi_{k}(\widetilde{H}_{a_ka_k})=\sum_{j=1}^{N_k} f'(\mu_k;\eta^k_j)Q^k_{b^k_j}(Q^k_{b^k_j})^\mathbb{T}$.
  Then, it is immediate to obtain
  \begin{align}\label{Phik}
  \Phi_{k}(\widetilde{H}_{a_k a_k})
   =\left\{\begin{array}{cl}
    \widetilde{H}_{a_k a_k} & {\rm if\;} k<r_0,\\
     \Pi_{\mathbb{S}^{|a_k|}_+}(\widetilde{H}_{a_ka_k}) & {\rm if\;} k=r_0, \\
     0 &{\rm if\;} k>r_0.
    \end{array}\right.
  \end{align}

  \noindent
  {\bf Step 2:} to calculate $I^{\mathbb{T}}_{a_k}\widetilde{H}G^{[2]}_{kl}(\Lambda(Z))\widetilde{H}I_{a_l}$.
  From the definition of $G^{[2]}_{kl}$, it follows that
  \begin{align}\label{Gkleq1}
  I^{\mathbb{T}}_{a_k}\widetilde{H}G^{[2]}_{kl}(\Lambda(Z))\widetilde{H}I_{a_l}
  &={\textstyle\sum_{j=1}^{r}} g^{[2]}(\mu_k,\mu_j,\mu_l)I^{\mathbb{T}}_{a_k}\widetilde{H}I_{a_j}I_{a_j}^\mathbb{T}
  \widetilde{H}I_{a_l}\nonumber\\
  &= P^{\mathbb{T}}_{a_k}H \big({\textstyle\sum_{j=1}^{r}} g^{[2]}(\mu_k,\mu_j,\mu_l)P_{a_j}P_{a_j}^\mathbb{T}\big) HP_{a_l}.
  \end{align}
  We make simplification for the term on the right hand side by the following three cases,
  where the fact that $g(\mu_i)=f(\mu_i)$ and $g'(\mu_i)=0$ for $i=1,2,\ldots,r$ are frequently used.

  \noindent
  {\bf Case 1: $k<l\leq  r_0$.} After an elementary calculation, it is easy to obtain that
  \[
  g^{[2]}(\mu_k,\mu_j,\mu_l)
   =\left\{\begin{array}{cl}
   \frac {1}{\mu_l-\mu_k} & {\rm if\;} \mu_j= \mu_k,\\
   \frac {1}{\mu_k-\mu_l} & {\rm if\;} \mu_j= \mu_l,\\
   0 & {\rm if\;}\mu_j\geq 0, \mu_j\neq \mu_k, \mu_j\neq \mu_l,\\
   \frac{\mu_j}{(\mu_k-\mu_j)(\mu_j-\mu_l)} &{\rm if\;} \mu_j<0.\\
   \end{array}\right.
  \]
  Together with equation \eqref{Gkleq1}, it is immediate to have that
  \begin{align}\label{GklResult1}
   I^{\mathbb{T}}_{a_k}\widetilde{H}G^{[2]}_{kl}(\Lambda(Z))\widetilde{H}I_{a_l}
   &= \frac {1}{\mu_l-\mu_k}\widetilde{H}_{a_ka_k}\widetilde{H}_{a_ka_l}+
   \frac {1}{\mu_k-\mu_l}\widetilde{H}_{a_ka_l}\widetilde{H}_{a_la_l}\nonumber\\
   &\quad+P^{\mathbb{T}}_{a_k}H \Big[\sum_{j=r_0+1}^{r}
   \frac{\mu_j}{(\mu_k-\mu_j)(\mu_j-\mu_l)}P_{a_j}P_{a_j}^\mathbb{T}\Big]HP_{a_l}.
   \end{align}

  \noindent
  {\bf Case 2: $k=r_0<l$.} In this case, $\mu_k=0$ and
  \(
  g^{[2]}(\mu_k,\mu_j,\mu_l)=\left\{\begin{array}{cl}
                              \!{1}/{(\mu_j\!-\!\mu_l)} & {\rm if\;} \mu_j>0,\\
                              0 & {\rm if\;} \mu_j\leq 0.
                                \end{array}\right.
  \)
  Together with equation \eqref{Gkleq1}, it is immediate to obtain that
  \begin{equation}\label{GklResult3}
   I^{\mathbb{T}}_{a_k}\widetilde{H}G^{[2]}_{kl}(\Lambda(Z))\widetilde{H}I_{a_l}
   =P^{\mathbb{T}}_{a_k}H \Big[\sum_{j=1}^{r_0-1} \frac{1}{(\mu_j-\mu_l)}P_{a_j}P_{a_j}^\mathbb{T}\Big]HP_{a_l}.
  \end{equation}

  \noindent
  {\bf Case 3: $r_0<k<l$.} Now,
  \(
    g^{[2]}(\mu_k,\mu_j,\mu_l)
    =\left\{\begin{array}{cl}
     \!\frac {\mu_j}{(\mu_j-\mu_k)(\mu_j-\mu_l)}& {\rm if}\ \mu_j\geq 0,\\
     0 & {\rm if}\ \mu_j<0.
     \end{array}\right.
  \)
  From \eqref{Gkleq1},
  \begin{align}\label{GklResult4}
   I^{\mathbb{T}}_{a_k}\widetilde{H}G^{[2]}_{kl}(\Lambda(Z))\widetilde{H}I_{a_l}
   =P^{\mathbb{T}}_{a_k}H \Big[\sum_{j=1}^{r_0-1} \frac{\mu_j}{(\mu_j-\mu_k)(\mu_j-\mu_l)}P_{a_j}P_{a_j}^\mathbb{T}\Big] HP_{a_l}
  \end{align}

  \noindent
  {\bf Step 3:} to characterize $B_{a_k a_l}$ for any $k\neq l$.
  By the eigenvalue index sets of $Z$, we have
  \begin{align}\label{gDifference}
  g^{[1]}(\Lambda(Z))=\left[\begin{matrix}
       \widehat{E}_{\alpha\alpha}  & E_{\alpha\beta} &\Sigma_{\alpha\gamma}\\
        E_{\alpha\beta}^\mathbb{T} & 0  & 0 \\
        \Sigma_{\alpha\gamma}^\mathbb{T}&  0 & 0
        \end{matrix}\right ]
  \end{align}
  where
  $\widehat{E}_{\alpha\alpha}=E_{\alpha\alpha}\!-\!{\rm Diag}(E_{a_1a_1},\cdots,E_{a_{r_0-1}a_{r_0-1}})$
  and $\Sigma_{ij}=\frac{\max\{0,\mu_i\}-\max\{0,\mu_j\}}{\mu_i-\mu_j}$
  for any $i\in \alpha$ and $j\in \beta$. Combining \eqref{gDifference}
  with \eqref{Phik}-\eqref{GklResult4} and \eqref{BKLFormula1}-\eqref{BKLFormula2},
  we get the result.
  \end{proof}
 \begin{proposition}\label{BKK}
  Let $f(t)=\max(0,t)$ for $t\in\mathbb{R}$. For each $k\in\{1,2,\ldots,r\}$,
  we have
  \begin{subequations}
  \begin{align}\label{BKKResult1}
   B_{a_k a_k}&=\widetilde{W}_{a_ka_k}-2P^{\mathbb{T}}_{a_k}H\Big[\sum_{j=r_0+1}^{r} \frac{\mu_j}{(\mu_j-\mu_k)^2}P_{a_j}P_{a_j}^\mathbb{T}\Big]HP_{a_k}\quad{\rm for}\ k<r_0;\\
  \label{BKKResult2}
   B_{a_{k} a_{k}}&=2P^{\mathbb{T}}_{a_k}H\Big[\sum_{j=1}^{r_0-1} \frac{\mu_j}{(\mu_j-\mu_k)^2}P_{a_j}P_{a_j}^\mathbb{T}\Big]HP_{a_k}\quad{\rm for}\ k>r_0;\\
  \label{BKKResult3}
   B_{a_k a_k}&=Q^k\!\left[\begin{matrix}
   (\widehat{V}_{k}(H,W))_{\alpha^k \alpha^k} & (\widehat{V}_{k}(H,W))_{\alpha^k \beta^k}
   &\Sigma_{\alpha^k\gamma^k}\circ (\widehat{V}_{k}(H,W))_{\alpha^k\gamma^k} \\
   *  & \Pi_{\mathbb{S}^{|\beta^k|}_+}((\widehat{V}_{k}(H,W))_{\beta^k\beta^k}) & 0\\
   * & 0 & 0
   \end{matrix}\right](Q^k)^\mathbb{T} \nonumber\\
   &\quad +2P^{\mathbb{T}}_{a_k}H \Big(\sum_{j=1}^{r_0-1} \frac{1}{\mu_j}P_{a_j}P_{a_j}^\mathbb{T}\Big)HP_{a_k}
   \quad{\rm for}\ k=r_0
  \end{align}
    where $\Sigma_{ij}:=\frac{\max\{0,\eta^k_i\}-\max\{0,\eta^k_j\}}{\eta^k_i-\eta^k_j}$
    for $i\in\alpha^k,j\in\beta^k$, and $\widehat{V}_{k}(H,W):=(Q^k)^{\mathbb{T}} \widetilde{V}_{k}(H,W) Q^k$.
  \end{subequations}
  \end{proposition}
  \begin{proof}
   We proceed the arguments by four steps as will be shown below.

  \noindent
  {\bf Step 1:} to calculate $\phi^{[1]}_{k}(\Lambda(\widetilde{H}_{a_k a_k}))$.
  By \eqref{phiketa}, for each $i,j\in\{1,\ldots,N_k\}$, we have
  \begin{align}
  \phi^{[1]}_k(\eta^k_i,\eta^k_j)=\left\{\begin{array}{cl} 1 & {\rm if\;} \mu_k>0,\eta^k_i\neq \eta^k_j \\
                                \frac{\max\{0,\eta^k_i\}-\max\{0,\eta^k_j\}}{\eta^k_i-\eta^k_j} & {\rm if\;} \mu_k=0, \eta^k_i\neq \eta^k_j  \\
                                0 &{\rm if\;}  \mu_k<0 {\ \ \rm or \ \ } \eta^k_i=\eta^k_j.
                                \end{array}\right.
  \end{align}
  Hence, with $\widehat{E}_{\alpha^k \alpha^k}:=E_{\alpha^k \alpha^k}-{\rm Diag}(E_{b^k_{1}b^k_{1}},\cdots, E_{b^k_{(N^0_k-1)}b^k_{(N^0_{k}-1)}})$, it holds that
  \begin{align}\label{phiDiff}
  \phi^{[1]}_{k}(\Lambda(\widetilde{H}_{a_k a_k}))
  =\left\{\begin{array}{cl}
   E_{a_ka_k}-{\rm Diag}(E_{b^k_{1} b^k_{1}},\cdots, E_{b^k_{N_k} b^k_{N_{k}}})
   & {\rm if\ \ } k<r_0,\\
   \left[\begin{matrix}
   \widehat{E}_{\alpha^k \alpha^k} & E_{\alpha^k \beta^k} &\Sigma_{\alpha^k\gamma^k} \\
   (E_{\alpha^k \beta^k})^\mathbb{T}  & 0  & 0\\
   \Sigma^{\mathbb{T}}_{\alpha^k\gamma^k}& 0 & 0
   \end{matrix}\right]
   & {\rm if\ \ } k=r_0,\\
   0 &{\rm if\ \ } k>r_0.
   \end{array}\right.
  \end{align}

  \noindent
  {\bf Step 2:} to calculate the term $\Psi_{k,p}(\widehat{V}^p_k(H,W))$ with $\widehat{V}^p_k(H,W)\!:=(Q^k_{b^k_p})^{\mathbb{T}}\widetilde{V}_k(H,W)Q^k_{b^k_p}$
  for each $p\in\{1,2,\ldots,N_k\}$. By the expression of $f$, it is easy to obtain that
  \begin{equation}\label{maxFSDir}
    f''(z;d,w)=\left\{\begin{array}{cl}
    w & {\rm if\;} z>0 {\;\rm or\;} z=0,d>0,\\
    0 & {\rm if\;} z<0 {\;\rm or\;} z=0,d<0, \\
    f(w) &{\rm if\;} z=d=0.
    \end{array}\right.
  \end{equation}
   Write $\psi_{k,p}(\cdot):=f''(\mu_k;\eta^k_p,\cdot)$.
   Then, for each $j\in\{1,2,\ldots,|b^k_p|\}$, it holds that
   \[
    \psi_{k,p}(\lambda_j(\widehat{V}^p_k(H,W)))
     =\left\{\begin{array}{cl}
     \lambda_j(\widehat{V}^p_k(H,W)) & {\rm if\;} \mu_k>0 {\;\rm or\;} \mu_k=0,\ \eta^k_p>0,\\
     0 & {\rm if\;} \mu_k<0 {\;\rm or\;} \mu_k=0,\eta^k_p<0, \\
     f(\lambda_j(\widehat{V}^p_k(H,W))) &{\rm if\;} \mu_k=\eta^k_p=0.
     \end{array}\right.
    \]
   Since $\Psi_{k,p}(\cdot)$ be the L$\rm \ddot{o}$wner operator associated to
   $\psi_{k,p}(\cdot)$, we obtain
  \begin{align}\label{PsiKp}
  \Psi_{k,p}(\widehat{V}^p_k(H,W))=\left\{\begin{array}{cl}
                      \widehat{V}^p_k(H,W) & {\rm if\;} k<r_0 {\ \ \rm or \ \ }k=r_0,\eta^{k}_{p}>0,\\
                        0 & {\rm if\;}{k>r_0} {\ \ \rm or\ \ }k=r_0,\eta^{k}_{p}<0,\\
                        \Pi_{\mathbb{S}^{|b^k_p|}_{+}}
                       (\widehat{V}^p_k(H,W))) &{\rm if\;} k=r_0,\eta^k_p=0.
                       \end{array}\right.
  \end{align}

  \noindent
  {\bf Step 3:} to characterize $I^{\mathbb{T}}_{a_k}\widetilde{H}G^{[2]}_{kk}(\Lambda(Z))\widetilde{H}I_{a_k}$
  for each $k\in\{1,2,\ldots,r\}$. Notice that
  \(
  g^{[2]}(\mu_k,\mu_j,\mu_k)=\left\{\begin{array}{cl}
                             \!\frac{f(\mu_j)-f(\mu_k)}{(\mu_j-\mu_k)^2}
                             &{\rm if\;} j\neq k,\\
                                0 & {\rm if\;}j= k.
                                \end{array}\right.
  \)
  Together with the equality \eqref{Gkleq1}, we obtain
  \begin{align}\label{Gkk}
   I^{\mathbb{T}}_{a_k}\widetilde{H}G^{[2]}_{kk}(\Lambda(Z))\widetilde{H}I_{a_k}
    &=P^{\mathbb{T}}_{a_k}H \Big[\sum_{j=1,j\neq k}^{r} \frac{f(\mu_j)-f(\mu_k)}{(\mu_j-\mu_k)^2}P_{a_j}P_{a_j}^\mathbb{T}\Big]HP_{a_k}.
  \end{align}

  \noindent
  {\bf Step 4:} to characterize $B_{a_k a_k}$. When $k<r_0$, we have $\mu_k>0$.
  By \eqref{Gkk}, we have
  \begin{align*}
  I^{\mathbb{T}}_{a_k}\widetilde{H}G^{[2]}_{kk}(\Lambda(Z))\widetilde{H}I_{a_k}
  &=P^{\mathbb{T}}_{a_k}H \Big[\sum_{j=1,j\neq k}^{r_0-1} \frac{P_{\!a_j}P_{\!a_j}^\mathbb{T}}{(\mu_j-\mu_k)}
  -\sum_{j=r_0}^{r} \frac{\mu_kP_{\!a_j}P_{\!a_j}^\mathbb{T}}{(\mu_j-\mu_k)^2}\Big]HP_{a_k}.
  \end{align*}
  By using \eqref{BKKFormula}, \eqref{phiDiff} and \eqref{PsiKp}
  and noting that $\widetilde{V}_k(H,W)=P^{\mathbb{T}}_{a_k}(W-2H(Z-\mu_kI)^\dagger H)P_{a_k}$,
  \begin{equation*}
  B_{a_k a_k}\!=\widetilde{V}_{k}(H,W)
  +2I_{a_k}^{\mathbb{T}}\widetilde{H}G^{[2]}_{kk}(\Lambda(Z))\widetilde{H}I_{a_k}
  =P^{\mathbb{T}}_{a_k}WP_{a_k}-2P^{\mathbb{T}}_{a_k}H\sum_{j=r_0}^{r} \frac{\mu_jP_{a_j}P_{a_j}^\mathbb{T}}{(\mu_j-\mu_k)^2}HP_{a_k}.
  \end{equation*}
  When $k>r_0$, we have $\mu_k<0$. From equation \eqref{Gkk}, it holds that
  \begin{align*}
  I^{\mathbb{T}}_{a_k}\widetilde{H}G^{[2]}_{kk}(\Lambda(Z))\widetilde{H}I_{a_k}
  &=P^{\mathbb{T}}_{a_k}H \Big(\sum_{j=1}^{r_0-1} \frac{\mu_j}{(\mu_j-\mu_k)^2}P_{a_j}P_{a_j}^\mathbb{T}\Big)HP_{a_k}.
  \end{align*}
  The result follows by \eqref{BKKFormula}, \eqref{phiDiff} and \eqref{PsiKp}.
  When $k=r_0$, we have $\mu_k=0$. From \eqref{Gkk},
  \begin{align*}
  I^{\mathbb{T}}_{a_k}\widetilde{H}G^{[2]}_{kk}(\Lambda(Z))\widetilde{H}I_{a_k}
  &=P^{\mathbb{T}}_{a_k}H \Big(\sum_{j=1}^{r_0-1} \frac{1}{\mu_j}P_{a_j}P_{a_j}^\mathbb{T}\Big)HP_{a_k}.
  \end{align*}
  The desired result for this case can be obtained by using
  \eqref{BKKFormula}, \eqref{phiDiff} and \eqref{PsiKp}.
 \end{proof}

 \section{Second-order tangent set to the SDCC set}\label{sec4}

  To characterize the second-order tangent set to the SDCC set $\Omega$,
  we need the following lemma, which states that the operator $\Pi_{\mathbb{S}^n_+}$
  is second-order directionally differentiable.
 \begin{lemma}\label{SOdir-Proj}
  The mappings $\Pi_{\mathbb{S}^n_+}(\cdot)$ and $\Pi_{\mathbb{S}^n_-}(\cdot)$
  are second-order directionally differentiable. Also,
  for any given $(X,Y)\in \Omega$ and $(F,G)\in \mathcal{T}_{\Omega}(X,Y)$,
  it holds that
  \[
    \Pi''_{\mathbb{S}^n_+}(X+Y;F+G,W)=W-\Pi''_{\mathbb{S}^n_-}(X+Y;F+G,W)
    \quad\ \forall\, W\in\mathbb{S}^n.
  \]
 \end{lemma}
 \begin{proof}
  Since $\Pi_{\mathbb{S}^n_+}(\cdot)$ and $\Pi_{\mathbb{S}^n_-}(\cdot)$
  are the L\"{o}wner operator associated to $t\mapsto \max(0,t)$
  and $t\mapsto\min(0,t)$ for $t\in\mathbb{R}$, respectively, the first part
  follows by \cite[Theorem 4.1]{ZhangZX13}. Write $Z=X+Y$.
  Since $(X,Y)\in \Omega$ and $(F,G)\in \mathcal{T}_{\Omega}(X,Y)$,
  by \cite[Corollary 3.1]{WuZZ14},
  \[
    \Pi'_{\mathbb{S}^n_-}(Z;F+G)=F+G-\Pi'_{\mathbb{S}^n_+}(Z;F+G).
  \]
  Together with $\Pi_{\mathbb{S}_{-}^n}(A)=A-\Pi_{\mathbb{S}_{+}^n}(A)$ for any $A\in\mathbb{S}^n$,
  it follows that
  \begin{align*}
   &\Pi_{\mathbb{S}^n_-}\Big(Z+t(F+G)+\frac{1}{2}t^2 W \Big)-\Pi_{\mathbb{S}^n_-}(Z)-t\Pi'_{\mathbb{S}^n_-}(Z;F+G)\\
   &=Z+t(F+G)+\frac{1}{2}t^2 W-\Pi_{\mathbb{S}^n_+}\Big(Z+t(F+G)+\frac{1}{2}t^2 W\Big)\\
   &\quad -Z+\Pi_{\mathbb{S}^n_+}(Z)-t(F+G)+t\Pi'_{\mathbb{S}^n_+}(Z;F+G))\\
   &=\frac{1}{2}t^2 W-\Big[\Pi_{\mathbb{S}^n_+}\Big(Z+t(F+G)+\frac{1}{2}t^2 W \Big)-\Pi_{\mathbb{S}^n_+}(Z)
     -t\Pi'_{\mathbb{S}^n_+}(Z;F+G))\Big].
  \end{align*}
  By the definition of second-order directional derivative,
  we obtain the result.
 \end{proof}

 Next we show that the SDCC set $\Omega$ is almost parabolically derivable,
 i.e., its outer second-order tangent set coincides with its inner second-order
 tangent set.
 \begin{proposition}\label{SOTangent}
  Fix an arbitrary $(X,Y)\in \Omega$. Then, for any given $(F,G)\in \mathcal{T}_{\Omega}(X,Y)$,
  \begin{align}\label{MD}
   &\mathcal{T}^{i,2}_{\Omega}((X,Y);(F,G))=\mathcal{T}^{2}_{\Omega}((X,Y);(F,G))\nonumber\\
   &=\Big\{(S,T)\in\mathbb{S}^n\times\mathbb{S}^n\,|\, \Pi''_{\mathbb{S}_+^n}(X+Y;F+G,S+T)=S\Big\}\\
   &=\Big\{(S,T)\in\mathbb{S}^n\times\mathbb{S}^n\,|\, \Pi''_{\mathbb{S}_-^n}(X+Y;F+G,S+T)=T\Big\}.\nonumber
  \end{align}
 \end{proposition}
 \begin{proof}
  The last equality follows by Lemma \ref{SOdir-Proj}. For the first two equalities,
  since $\mathcal{T}^{i,2}_{\Omega}((X,Y);(F,G))\subseteq\mathcal{T}^{2}_{\Omega}((X,Y);(F,G))$,
  it suffice to establish the following inclusions
  \begin{equation}\label{aim-inclusion}
   \mathcal{T}^{2}_{\Omega}((X,Y);(F,G))\subset\mathscr{D}\subset \mathcal{T}^{i,2}_{\Omega}((X,Y);(F,G))
  \end{equation}
  where $\mathscr{D}$ is the set in \eqref{MD}.
  By the definition of the set $\Omega$, for any $(X',Y')\in\mathbb{S}^n\times\mathbb{S}^n$,
  \begin{equation}\label{ProjSDP}
   (X',Y')\in \Omega\Longleftrightarrow \Pi_{\mathbb{S}^n_+}(X'+Y')=X'
   \Longleftrightarrow \Pi_{\mathbb{S}^n_+}(X'+Y')=Y'.
  \end{equation}
  In addition, by \cite[Theorem 3.1]{WuZZ14},
  for any $(X',Y')\in\mathbb{S}^n\times\mathbb{S}^n$
  and $(F',G')\in\mathbb{S}^n\times\mathbb{S}^n$,
  \begin{subequations}
  \begin{align}\label{tangent-cone1}
   (F',G')\in \mathcal{T}_{\Omega}(X',Y')
   &\Longleftrightarrow \Pi'_{\mathbb{S}^n_+}(X'+Y';F'+G')=F'\\
   \label{tangent-cone2}
   &\Longleftrightarrow\Pi'_{\mathbb{S}^n_-}(X'+Y';F'+G')=G'.
  \end{align}
  \end{subequations}
  Take $(S,T)\in\!\mathcal{T}^{2}_{\Omega}((X,Y);(F,G))$.
  By the definition of outer second-order tangent sets, there are sequences
  $t_k\downarrow 0$ and $(S^k,T^k)\to(S,T)$ such that
  \begin{equation*}
    (X,Y)+t_k(F,G)+\frac{1}{2}t_k^2(S^k,T^k)\in \Omega\quad\ \forall k.
  \end{equation*}
  Combining with $(X,Y)\in\Omega$ and $(F,G)\in \mathcal{T}_{\Omega}(X,Y)$
  and using \eqref{ProjSDP}-\eqref{tangent-cone2}, we obtain
  \begin{align*}
  &\Pi_{\mathbb{S}^n_+}\Big(X+Y+t_k(F+G)+\frac{1}{2}t^2_k(S^k+T^k)\Big)
    =X+t_k F+\frac{1}{2}t_k^2 S^k\\
  &=\Pi_{\mathbb{S}^n_+}(X+Y)+t_k\Pi'_{\mathbb{S}^n_+}(X+Y;F+G)+\frac{1}{2}t_k^2S^k.
  \end{align*}
  By the second-order directional differentiability and the Lipschitz
  continuity of $\Pi_{\mathbb{S}^n_+}$, from the last equality we obtain
  $\Pi''_{\mathbb{S}^n_+}(X+Y;F+G,S+T)=S$, i.e., $(S,T)\in\mathscr{D}$.
  By the arbitrariness of $(S,T)$ in $\mathcal{T}^{2}_{\Omega}((X,Y);(F,G))$,
  it follows that $\mathcal{T}^{2}_{\Omega}((X,Y);(F,G))\subset\mathscr{D}$.

  \medskip

  Next, take an arbitrary $(S,T)\in\mathscr{D}$. For any sufficiently small $t>0$, write
  \[
    Z(t):=X+Y+t(F+G)+\frac{t^2}{2}(S+T),\,X(t):=\Pi_{\mathbb{S}^n_+}(Z(t))
    \ \ {\rm and}\ \ Y(t):=\Pi_{\mathbb{S}^n_-}(Z(t)).
  \]
  Then $(X(t),Y(t))\in\Omega$. By the second-order directional differentiability
  of $\Pi_{\mathbb{S}^n_+}$ and $\Pi_{\mathbb{S}_{-}^n}$,
 \begin{subnumcases}{}\label{Xt-equa}
   X(t)-X-t\Pi_{\mathbb{S}^n_+}'(X+Y;F+G)-\frac{t^2}{2}\Pi''_{\mathbb{S}^n_+}(X+Y;F+G,S+T)
   =o(t^2),\qquad\\
   \label{Yt-equa}
   Y(t)-Y-t\Pi_{\mathbb{S}_{-}^n}'(X+Y;F+G)-\frac{t^2}{2}\Pi''_{\mathbb{S}_{-}^n}(X+Y;F+G,S+T)
   =o(t^2).\qquad
  \end{subnumcases}
  Together with $(F,G)=\big(\Pi_{\mathbb{S}^n_+}'(X+Y;F+G),\Pi_{\mathbb{S}_{-}^n}'(X+Y;F+G)\big)$
  by $(F,G)\in\mathcal{T}_{\Omega}(X,Y)$
  and $(S,T)=\big(\Pi''_{\mathbb{S}^n_+}(X+Y;F+G,S+T),\Pi''_{\mathbb{S}_{-}^n}(X+Y;F+G,S+T)\big)$,
  it holds that
  \[
    \lim_{t\downarrow 0}\frac{X(t)-X-t G-\frac{t^2}{2}S}{t^2}=0
    \ \ {\rm and}\ \
    \lim_{t\downarrow 0}\frac{Y(t)-Y-t H-\frac{t^2}{2}T}{t^2}=0.
  \]
  Since $(X(t),Y(t))\in \Omega$, this means that
  ${\rm dist}\big((X,Y)+t(F,G)+\frac{t^2}{2}(S,T),\Omega\big)=o(t^2)$,
  which is equivalent to $(S,T)\in \mathcal{T}^{i,2}_{\Omega}((X,Y);(F,G))$.
  By the arbitrariness of $(S,T)$ in the set $\mathscr{D}$, it follows that
  $\mathscr{D}\subset \mathcal{T}^{i,2}_{\Omega}((X,Y);(F,G))$.
  Thus, the first two equalities hold.
  \end{proof}

  The following theorem gives a characterization of the second-order tangent set to $\Omega$.
  \begin{theorem}\label{SO-tangent}
   Fix an arbitrary $(X,Y)\in \Omega$ and an arbitrary $(F,G)\in \mathcal{T}_{\Omega}(X,Y)$.
   Write $Z=X+Y$ and $H=F+G$. Let $Z$ have the spectral decomposition as
   in \eqref{Zdecom}-\eqref{idx-eigZ}, and let $\widetilde{H}_{a_ka_k}$
   for $k\in\{1,2,\ldots,r\}$ have the spectral decomposition as in
   \eqref{Hkdecom}-\eqref{alpk-gamk}. Then,
   $(P \widetilde{S} P^\mathbb{T},P \widetilde{T} P^\mathbb{T})
   \in\mathcal{T}^{2}_{\Omega}((X,Y);(F,G))$ if and only if
   for each $k,l\in \{1,\ldots,r\}$ with $k<l$,
  \begin{subequations}\label{TKL}
  \begin{align}
  \label{TKLResult1}
  \widetilde{T}_{a_k a_l}=-2P^{\mathbb{T}}_{a_k}H\Big[\sum\limits_{j=r_0+1}^{r} \frac{\mu_j}{(\mu_j-\mu_k)(\mu_j-\mu_l)}P_{\!a_j}P_{\!a_j}^\mathbb{T}\Big]HP_{a_l}
   \quad {\rm for}\ k<l<r_0;\quad\\
  \label{TKLResult2}
  \widetilde{T}_{a_k a_l}
   =2P^{\mathbb{T}}_{a_k}H\!\sum\limits_{j=r_0+1}^{r}\frac{P_{\!a_j}P_{\!a_j}^\mathbb{T}}{\mu_k-\mu_j}HP_{a_l}
    -\frac{2}{\mu_k}\widetilde{H}_{a_k a_l} \Pi_{\mathbb{S}^{|a_l|}_-}(\widetilde{H}_{a_l a_l})
  \quad {\rm for}\ k<l=r_0; \\
  \label{TKLResult3}
  \widetilde{S}_{a_k a_l}=2P^{\mathbb{T}}_{a_k}H\!\sum\limits_{j=1}^{r_0-1} \frac{P_{\!a_j}P_{\!a_j}^\mathbb{T}}{(\mu_j-\mu_l)}HP_{a_l}
  -\frac{2}{\mu_l}\Pi_{\mathbb{S}^{|a_l|}_+}(\widetilde{H}_{a_k a_k})\widetilde{H}_{a_k a_l}
  \quad{\rm for}\ k=r_0<l; \\
  \label{TKLResult4}
  \widetilde{S}_{a_k a_l}=2P^{\mathbb{T}}_{a_k}H\Big[\sum\limits_{j=1}^{r_0-1} \frac{\mu_j}{(\mu_j-\mu_k)(\mu_j-\mu_l)}P_{\!a_j}P_{\!a_j}^\mathbb{T}\Big]HP_{a_l}
  \quad{\rm for}\ r_0<k<l; \qquad\\
  \label{TKLResult5}
  \frac{\mu_l\widetilde{S}_{a_k a_l}}{\mu_k-\mu_l}+ \frac{\mu_k\widetilde{T}_{a_ka_l}}{\mu_k-\mu_l}
   +2I^\mathbb{T}_{a_k} \widetilde{H}G^{[2]}_{kl}(\Lambda(Z))\widetilde{H}I_{a_l}+
   \frac{2\widetilde{H}_{a_k a_k}\widetilde{H}_{a_k a_l}}{\mu_k-\mu_l}=0
  \quad{\rm for}\ k<r_0<l
  \end{align}
  \end{subequations}
  where $I^\mathbb{T}_{a_k} \widetilde{H}G^{[2]}_{kl}(\Lambda(Z))\widetilde{H}I_{a_l}$
  is same as the one in \eqref{klr05}, and $(\widetilde{S}_{a_ka_k},\widetilde{T}_{a_k a_k})$
  satisfies
  \begin{subequations}\label{TKK}
   \begin{align}\label{TKKResult1}
   \widetilde{T}_{a_k a_k}=2P^{\mathbb{T}}_{a_k}H\Big[\sum_{j=r_0+1}^{r} \frac{\mu_j}{(\mu_j-\mu_k)^2}P_{a_j}P_{a_j}^\mathbb{T}\Big]HP_{a_k}
   \quad {\rm for}\ k<r_0;\qquad\qquad\qquad\quad \\
   \label{TKKResult3}
   \widetilde{S}_{a_k a_k}=2P^{\mathbb{T}}_{a_k}H\Big[\sum_{j=1}^{r_0-1} \frac{\mu_j}{(\mu_j-\mu_k)^2}P_{a_j}P_{a_j}^\mathbb{T}\Big]HP_{a_k}
   \quad {\rm for}\ k>r_0;\qquad\qquad\qquad\qquad \\
   \label{TKKResult2}
   \!\left\{\begin{array}{ll}
   \widehat{S}_{a_ka_k}:=(Q^k)^{\mathbb{T}}\widetilde{S}_{a_ka_k}Q^k
   =\left[\begin{matrix}
   (\widehat{S}_{a_ka_k})_{\alpha^k\alpha^k} & (\widehat{S}_{a_ka_k})_{\alpha^k\beta^k} & (\widehat{S}_{a_ka_k})_{\alpha^k\gamma^k} \\
   *  & (\widehat{S}_{a_ka_k})_{\beta^k\beta^k} & \Delta^+_{\beta^k\gamma^k}\\
   * & * & \Delta^+_{\gamma^k\gamma^k}
   \end{matrix}\right],\\
   \widehat{T}_{a_ka_k}:=(Q^k)^{\mathbb{T}}\widetilde{T}_{a_ka_k}Q^k
   =\left[\begin{matrix}
   \Delta^{-}_{\alpha^k\alpha^k} & \Delta^{-}_{\alpha^k\beta^k} &(\widehat{T}_{a_ka_k})_{\alpha^k\gamma^k} \\
   *  & (\widehat{T}_{a_ka_k})_{\beta^k\beta^k} & (\widehat{T}_{a_ka_k})_{\beta^k\gamma^k}\\
   * & * & (\widehat{T}_{a_ka_k})_{\gamma^k\gamma^k}
   \end{matrix}\right],\quad{\rm for}\ k=r_0\\
  (\Sigma_{\alpha^k\gamma^k}\!-\!E_{\alpha^k\gamma^k})\!\circ\!
  (\widehat{S}_{a_ka_k}\!-\!\Delta^+)_{\alpha^k\gamma^k}
   +\Sigma_{\alpha^k\gamma^k}\!\circ\!(\widehat{T}_{a_ka_k}-\Delta^-)_{\alpha^k\gamma^k}\!=0,\\
   (\widehat{S}_{a_ka_k}-\Delta^+)_{\beta^k\beta^k}=\Pi_{\mathbb{S}_{+}^{|\beta^k|}}
   \big[(\widehat{S}_{a_ka_k}+\widehat{T}_{a_ka_k}
   -\Delta^+-\Delta^-)_{\beta^k\beta^k}\big]
   \end{array}\right.
  \end{align}
  with $\Delta^{+}\!:=2(Q^k)^{\mathbb{T}}P^{\mathbb{T}}_{a_k}H\!{\displaystyle\sum_{j=1}^{r_0-1}}\frac{P_{a_j}
   P^{\mathbb{T}}_{a_j}}{\mu_j}HP_{a_k}Q^k$ and
  $\Delta^{-}\!:=2(Q^k)^{\mathbb{T}}P^{\mathbb{T}}_{a_k}H\!{\displaystyle\sum_{j=r_0+1}^{r}}\!\frac{P_{a_j}
  P^{\mathbb{T}}_{a_j}}{\mu_j}HP_{a_k}Q^k$.
  \end{subequations}
  \end{theorem}
  \begin{proof}
  By Proposition \ref{SOTangent}, $(S,T)\in \mathcal{T}^{2}_{\Omega}((X,Y);(F,G))$
  iff $\Pi''_{\mathbb{S}^n_+}(Z;H,S+T)=S$.
  By invoking Lemma \ref{ReviseTheo} and Proposition \ref{BKL}-\ref{BKK} with $W=S+T$,
  it follows that
  \[
    S=\Pi''_{\mathbb{S}^n_+}(Z;H,W)=P B P^{T}
  \]
  where $B_{a_ka_l}$ for $k,l\in\{1,\ldots,r\}$ has the form
  in \eqref{klr01}-\eqref{klr05} or \eqref{BKKResult1}-\eqref{BKKResult3}.
  An elementary calculation shows that
  $(S,T)\in \mathcal{T}^{2}_{\Omega}((X,Y);(F,G))$ iff $(S,T)$ satisfies \eqref{TKL} and \eqref{TKK}.
 \end{proof}

  Theorem \ref{SO-tangent} provides an exact characterization for the second-order
  tangent set to the SDCC set $\Omega$. For some special $(X,Y)\in\Omega$ and
  $(F,G)\in\mathcal{T}_{\Omega}(X,Y)$, Lemma \ref{Maincor} in Appendix illustrates
  the structure of the corresponding $\mathcal{T}_{\Omega}^2((X,Y);(F,G))$.

 \section{Second-order optimality conditions for SDCMPCC}\label{sec5}

  We shall derive the second-order optimality conditions of \eqref{SDCMPCC}
  in terms of the second-order tangent set to $\Omega$.
  For convenience, write $\Upsilon(x):=(h(x);\Theta(x))$ for $x\in\mathbb{X}$
  and denote by $\mathcal{S}$ the feasible set of \eqref{SDCMPCC}.
  For a given $x\in\mathcal{S}$, the critical cone of \eqref{ESDCMPCC} at $x$
  takes the form of
  \[
   \mathcal{C}(x):=\Big\{d\in \mathbb{X}\ |\ \varphi'(x)d\le 0,\,\Upsilon'(x)d\in\mathcal{T}_{K\times\Omega}(\Upsilon(x))\Big\}.
  \]
  Let $L\!:\mathbb{X}\times\mathbb{Y}\times(\mathbb{S}^n\times\mathbb{S}^n)\to\mathbb{R}$
  denote the Lagrange function of \eqref{ESDCMPCC}, which is defined as
  \[
    L(x,\xi,\Gamma):=\varphi(x)+\langle \xi,h(x)\rangle+\langle \Gamma,\Theta(x)\rangle
    =\varphi(x)+\langle(\xi,\Gamma),\Upsilon(x)\rangle.
  \]
  For a given $x\in\mathcal{S}$, the following multiplier sets of \eqref{ESDCMPCC}
  associated to $x$ are needed:
  \begin{subequations}
  \begin{align}\label{Tanget-Mset}
   \mathcal{M}^{c}(x)&:=\Big\{(\xi,\Gamma)\in\mathbb{Y}\times(\mathbb{S}_{+}^n\times\mathbb{S}_{+}^n)
   \ |\ \nabla_x L(x,\xi,\Gamma)=0,\,(\xi,\Gamma)\in\mathcal{N}_{K\times\Omega}^c(\Upsilon(x))\Big\},\\
   \label{Limit-Mset}
   \mathcal{M}(x)&:=\Big\{(\xi,\Gamma)\in\mathbb{Y}\times(\mathbb{S}_{+}^n\times\mathbb{S}_{+}^n)
   \ |\ \nabla_x L(x,\xi,\Gamma)=0,\,(\xi,\Gamma)\in\mathcal{N}_{K\times\Omega}(\Upsilon(x))\Big\},\\
   \label{Regular-Mset}
   \widehat{\mathcal{M}}(x)&:=\Big\{(\xi,\Gamma)\in\mathbb{Y}\times(\mathbb{S}_{+}^n\times\mathbb{S}_{+}^n)
   \ |\ \nabla_x L(x,\xi,\Gamma)=0,\, (\xi,\Gamma)\in\widehat{\mathcal{N}}_{K\times\Omega}(\Upsilon(x))\Big\}.
   \end{align}
 \end{subequations}
  Clearly, $\widehat{\mathcal{M}}(x)\subset \mathcal{M}(x)\subset \mathcal{M}^c(x)$
  and the multiplier sets $\widehat{\mathcal{M}}(x)$ and $\mathcal{M}^c(x)$ are convex.
  \begin{theorem}\label{SONC}
   Let $x^*$ be a locally optimal solution of \eqref{SDCMPCC}.
   Suppose that the multifunction $\mathcal{F}(z)\!:=\Upsilon(z)-K\times\Omega$
   is metrically subregular at $x^*$ for the origin.
   Then, $\mathcal{M}(x^*)\ne\emptyset$, and moreover, for any $d\in\mathcal{C}(x^*)$
   and any nonempty convex set $T(d)\subseteq\mathcal{T}^2_{K\times\Omega}(\Upsilon(x^*);\Upsilon'(x^*)d)$
   with ${\rm ri}\big[T(d)+{\mathcal{T}}^c_{K\times\Omega}(\Upsilon(x^*))\big]
   \cap\big[{\rm range}(\Upsilon'(x^*))\!+\!\nabla^2\Upsilon(x^*)(d,d)\big]\ne\emptyset$,
   it holds that
   \[
    \sup_{(\xi,\Gamma)\in\mathcal{M}^{c}(x^*)}\Big\{\langle d,\nabla^2_{xx}L(x^*,\xi,\Gamma)d\rangle
    -\sigma((\xi,\Gamma)\,|\,T(d))\Big\}\geq 0.
   \]
  \end{theorem}
 \begin{proof}
  Since $x^*$ is a locally optimal solution of the problem \eqref{SDCMPCC}, it holds that
  \begin{equation}\label{fo-optimality}
    0\in\widehat{\partial}(\varphi+\delta_{K\times\Omega}\circ\Upsilon)(x^*)
    =\nabla\varphi(x^*)+\widehat{\partial}(\delta_{K\times\Omega}\circ\Upsilon)(x^*)
    \subseteq \nabla\varphi(x^*)+\partial(\delta_{K\times\Omega}\circ\Upsilon)(x^*).
  \end{equation}
  Since $\mathcal{F}$ is metrically subregular at $x^*$ for the origin,
  from \cite[Page 211]{Ioffe08} it follows that
  \[
    \partial(\delta_{K\times\Omega}\circ\Upsilon)(x^*)
    \subseteq \nabla\Upsilon(x^*)\mathcal{N}_{K\times\Omega}(\Upsilon(x^*)).
  \]
  The last two equations imply that there exists
  $(\xi^*,\Gamma^*)\in\mathcal{N}_{K\times\Omega}(\Upsilon(x^*))$
  such that
  \[
    0=\nabla\varphi(x^*)+\nabla\Upsilon(x^*)(\xi^*;\Gamma^*).
  \]
  That is, $(\xi^*,\Gamma^*)\in\mathcal{M}(x^*)$,
  and the set $\mathcal{M}(x^*)$ is nonempty. The first part follows.

  \medskip

  For the second part, fix an arbitrary $d\in\mathcal{C}(x^*)$ and let
  $\overline{T}(d):={\rm cl}\big[T(d)+{\mathcal{T}}^c_{K\times\Omega}(\Upsilon(x^*))\big]$.
  Obviously, the set $\overline{T}(d)$ is a nonempty closed and convex set.
  By \cite[Proposition 3.12]{RW98},
  \begin{equation}\label{ctangent-equa}
   {\mathcal{T}}^2_{K\times\Omega}(\Upsilon(x^*);\Upsilon'(x^*)d)
   +{\mathcal{T}}^c_{K\times\Omega}(\Upsilon(x^*))
   ={\mathcal{T}}^2_{K\times\Omega}(\Upsilon(x^*);\Upsilon'(x^*)d).
  \end{equation}
   Together with $T(d)\subset {\mathcal{T}}^2_{K\times\Omega}(\Upsilon(x^*);\Upsilon'(x^*)d)$
   and the closedness of ${\mathcal{T}}^2_{K\times\Omega}(\Upsilon(x^*);\Upsilon'(x^*)d)$,
   we have $\overline{T}(d)\subset {\mathcal{T}}^2_{K\times\Omega}(\Upsilon(x^*);\Upsilon'(x^*)d)$.
   Since $x^*$ be a locally optimal solution of \eqref{SDCMPCC},
   by Lemma \ref{MSCQ-result}(i) and the definition of the second-order tangent set,
   it is not hard to obtain
   \[
    \varphi'(x^*)w+\nabla^2 \varphi(x^*)(d,d)\geq 0\quad\ \forall\,d\in\mathcal{C}(x^*), w\in {\mathcal{T}}^2_{\mathcal{S}}(x^*,d).
   \]
   By Lemma \ref{MSCQ-result}(ii), $w\in {\mathcal{T}}^2_{\mathcal{S}}(x^*,d)$
   iff $\Upsilon'(x^*)w+\nabla^2\Upsilon(x^*)(d,d)\in {\mathcal{T}}^2_{K\times\Omega}(\Upsilon(x^*);\Upsilon'(x^*)d)$.
   Therefore, for any  $d\in\mathcal{C}(x^*)$, the following problem has a nonnegative optimal value:
   \begin{align*}
    &\inf_{w\in\mathbb{X}}\ \varphi'(x^*)w+\nabla^2 \varphi(x^*)(d,d) \\
    &\ {\rm s.t.}\ \ \Upsilon'(x^*)w+\nabla^2\Upsilon(x^*)(d,d)\in {\mathcal{T}}^2_{K\times\Omega}(\Upsilon(x^*);\Upsilon'(x^*)d).
   \end{align*}
   Recall that $\overline{T}(d)\subset {\mathcal{T}}^2_{K\times\Omega}(\Upsilon(x^*);\Upsilon'(x^*)d)$.
   The following convex minimization problem
   \begin{align}\label{ConvexPro}
   &\inf_{w\in\mathbb{X}}\ \varphi'(x^*)w+\nabla^2 \varphi(x^*)(d,d) \nonumber\\
   &\ {\rm s.t.}\ \ \Upsilon'(x^*)w+\nabla^2\Upsilon(x^*)(d,d)\in \overline{T}(d)
   \end{align}
   has a nonnegative optimal value. An elementary calculation yields
   the dual of \eqref{ConvexPro} as
   \begin{equation}\label{dual1}
    \max_{(\xi,\Gamma)\in\mathbb{Y}\times\mathbb{S}^n\times\mathbb{S}^n}
    \Big\{\langle d,\nabla^2_{xx}L(x^*,\xi,\Gamma)d\rangle-\sigma((\xi,\Gamma)\,|\,\overline{T}(d))
    \ \ {\rm s.t.}\ \ \nabla_xL(x^*,\xi,\Gamma)=0\Big\}.
   \end{equation}
   It is not hard to argue that
   $\sigma((\xi,\Gamma)\,|\,\overline{T}(d))=+\infty$ for all
   $(\xi,\Gamma)\notin\mathcal{N}_{K\times\Omega}^c(\Upsilon(x^*))$. (Indeed,
   by fixing an arbitrary $(\xi,\Gamma)\notin\mathcal{N}_{K\times\Omega}^c(\Upsilon(x^*))$,
   there exists $(\eta,W)\in{\mathcal{T}}^c_{K\times\Omega}(\Upsilon(x^*))$ such that
   $\langle(\xi,\Gamma),(\eta,W)\rangle>0$. Take an $(\eta^0,S^0)\in T(d)$
   since $T(d)\ne\emptyset$. Clearly, for any $k>0$, $(\eta^0,S^0)+k(\eta,W)\in\overline{T}(d)$.
   Then, $\sigma((\xi,\Gamma)\,|\,\overline{T}(d))\ge\langle(\xi,\Gamma),(\eta^0,S^0)+k(\eta,W)\rangle$,
  which implies that $\sigma((\xi,\Gamma)\,|\,\overline{T}(d))=+\infty$.)
  So, the problem \eqref{dual1} can be equivalently written as
  \begin{equation}\label{dual2}
   \sup_{(\xi,\Gamma)\in\mathcal{M}^{c}(x^*)}\Big\{\langle d,\nabla^2_{xx}L(x^*,\xi,\Gamma)d\rangle
   -\sigma((\xi,\Gamma)\,|\,\overline{T}(d))\Big\}.
  \end{equation}
  Since $\big[{\rm range}(\Upsilon'(x^*))+\!\nabla^2\Upsilon(x^*)(d,d)\big]\cap{\rm ri}\big[\overline{T}(d)\big]\ne\emptyset$ by the given condition,
  there is no dual gap between \eqref{ConvexPro} and its dual \eqref{dual2}
  by \cite[Corollary 31.2.1(a)]{Roc70}, which means that
  \[
   \sup_{(\xi,\Gamma)\in\mathcal{M}^{c}(x^*)}\Big\{\langle d,\nabla^2_{xx}L(x^*,\xi,\Gamma)d\rangle
   -\sigma((\xi,\Gamma)\,|\,\overline{T}(d))\Big\}\geq 0.
  \]
  This, together with $T(d)\subseteq\overline{T}(d)$, implies the second part of the conclusions.
 \end{proof}

 When the subregularity of $\mathcal{F}$ in Theorem \ref{SONC} is respectively
 strengthened to be the metric regularity and constraint nondegeneracy,
 we have the following corollaries.
 \begin{corollary}\label{SONC1}
  Let $x^*$ be a locally optimal solution of the problem \eqref{ESDCMPCC}. If
  the multifunction $\mathcal{F}$ in Theorem \ref{SONC} is
  metrically regular at $x^*$ for the origin, then for any $d\in\mathcal{C}(x^*)$
  and any nonempty convex set $\Sigma(d)\subseteq\mathcal{T}^2_{\Omega}(\Theta(x^*);\Theta'(x^*)d)$
  it holds that
  \[
   \sup_{(\xi,\Gamma)\in\mathcal{M}^{c}(x^*)}\Big\{\langle d,\nabla^2_{xx}L(x^*,\xi,\Gamma)d\rangle
   -\sigma\big((\xi,\Gamma)\,|\,\mathcal{T}^2_{K}(h(x^*);h'(x^*)d)\times\Sigma(d)\big)\!\Big\}\ge 0.
  \]
 \end{corollary}
 \begin{proof}
  The metric regularity of $\mathcal{F}$ at $x^*$ for the origin is equivalent to Robinson's CQ, i.e., $\mathcal{N}_{K\times\Omega}(\Upsilon(x^*))\cap{\rm Ker}(\nabla\Upsilon(x^*))=\{0\}$,
  which implies that
  \begin{equation}\label{tangent-CQ}
   \Upsilon'(x^*)\mathbb{X}
    +\mathcal{T}_{K\times\Omega}^c(\Upsilon(x^*))
    =\left(\begin{matrix}
     \mathbb{Y}\\
     \mathbb{S}^n\times\mathbb{S}^n
     \end{matrix}\right).
  \end{equation}
  Fix an arbitrary $d\in\mathcal{C}(x^*)$.
  For each $(w,V)\in\mathcal{T}^2_{K}(h(x^*);h'(x^*)d)\times\Sigma(d)$,
  by invoking \eqref{tangent-CQ} with $-(w,V)+\nabla^2\Upsilon(x^*)(d,d)$,
  there exist $\eta\in\mathbb{X}$ and
  $(\xi,\Gamma)\in\mathcal{T}_{K\times\Omega}^c(\Upsilon(x^*))$ such that
  $\Upsilon'(x^*)\eta+\nabla^2\Upsilon(x^*)(d,d)=(w+\xi,V+\Gamma)$.
  This implies that
  \[
   \big[{\rm range}(\Upsilon'(x^*))\!+\!\nabla^2\Upsilon(x^*)(d,d)\big]\cap
   {\rm ri}\big[\mathcal{T}^2_{K}(h(x^*);h'(x^*)d)\times\Sigma(d)
   +{\mathcal{T}}^c_{K\times\Omega}(\Upsilon(x^*))\big]\ne\emptyset.
  \]
  Take $T(d)=\mathcal{T}^2_{K}(h(x^*);h'(x^*)d)\times\Sigma(d)$.
  The result follows by Theorem \ref{SONC}.
 \end{proof}
 \begin{corollary}\label{SONC2}
  Let $x^*$ be a locally optimal solution of the problem \eqref{SDCMPCC}. Suppose that
  \begin{equation}\label{SRCQ}
   {\rm Sp}\big(\mathcal{N}_{K\times\Omega}(\Upsilon(x^*))\big)
   \cap{\rm Ker}(\nabla\Upsilon(x^*))=\{0\}.
  \end{equation}
  Then, $\widehat{\mathcal{M}}(x^*)=\mathcal{M}(x^*)=\mathcal{M}^c(x^*)$
  is a singleton, say $\{(\xi^*,\Gamma^*)\}$, and for any $d\in\mathcal{C}(x^*)$,
  \[
   \langle d,\nabla^2_{xx}L(x^*,\xi^*,\Gamma^*)d\rangle
   -\sigma\big((\xi^*,\Gamma^*)\,|\,\mathcal{T}^2_{K\times\Omega}(\Upsilon(x^*);\Upsilon'(x^*)d)\big)\geq 0.
  \]
 \end{corollary}
 \begin{proof}
  We first argue that $\widehat{\mathcal{M}}(x^*)\ne\emptyset$.
  Since $[{\rm Sp}(\mathcal{N}_{K\times\Omega}(\Upsilon(x^*)))]^{\circ}
  ={\rm lin}[\mathcal{T}_{K\times\Omega}^{c}(\Upsilon(x^*))]$,
  taking the negative polar to the both sides of equation \eqref{SRCQ} yields that
  \[
    \Upsilon'(x^*)\mathbb{X}+{\rm lin}[\mathcal{T}_{K\times\Omega}^{c}(\Upsilon(x^*))]
    =\left(\begin{matrix}
      \mathbb{Y}\\
      \mathbb{S}^n\times\mathbb{S}^n
      \end{matrix}\right).
  \]
  Clearly, this condition implies that the multifunction $\mathcal{F}$
  in Theorem \ref{SONC1} is subregular at $x^*$ for the origin.
  From Lemma \ref{atangent} in Appendix,
  $\mathcal{T}_{\Omega}(\Theta(x^*))+\mathcal{T}_{\Omega}^{c}(\Theta(x^*))
  =\mathcal{T}_{\Omega}(\Theta(x^*))$. Together with the convexity of $K$,
  $\mathcal{T}_{K\times\Omega}(\Upsilon(x^*))+{\rm lin}[\mathcal{T}_{K\times\Omega}^{c}(\Upsilon(x^*))]
  =\mathcal{T}_{K\times\Omega}(\Upsilon(x^*))$.
  From \cite[Theorem 4]{Gfrerer16} it follows that
  $\widehat{\mathcal{N}}_{\mathcal{S}}(x^*)
  =\nabla\Upsilon(x^*)\widehat{\mathcal{N}}_{K\times\Omega}(\Upsilon(x^*))$.
  Thus, by \eqref{fo-optimality}, we have $\widehat{\mathcal{M}}(x^*)\ne\emptyset$.
  Consequently, $\mathcal{M}^c(x^*)\supseteq\mathcal{M}(x^*)\ne \emptyset$. In addition, it is easy to check that
  the condition \eqref{SRCQ} implies that $\mathcal{M}(x^*)$ is a singleton. So,
  we have $\widehat{\mathcal{M}}(x^*)=\mathcal{M}(x^*)$ is a singleton.
  Recall that $\mathcal{N}_{K\times\Omega}^{c}(\Upsilon(x^*))
  ={\rm cl}[{\rm co}\mathcal{N}_{K\times\Omega}(\Upsilon(x^*))]$.
  It is easy to verify that
  \[
   {\rm Sp}\big(\mathcal{N}_{K\times\Omega}^{c}(\Upsilon(x^*))\big)\cap{\rm Ker}(\nabla\Upsilon(x^*))=\{0\}.
  \]
  This implies that $\mathcal{M}^c(x^*)$ is a singleton.
  Applying Corollary \ref{SONC1} with $\Sigma(d)=\{V\}$
  for each $V\in\mathcal{T}^2_{\Omega}(\Theta(x^*);\Theta'(x^*)d))$
  and the singleton of $\mathcal{M}(x^*)$ yields the second part.
 \end{proof}
 \begin{remark}
  Recently, Gfrerer et al. derived a dual form of second-order necessary conditions
  by the lower generalized support function under the directional metric subregularity
  (see \cite[Theorem 2]{Gfrerer19}). Although the condition \eqref{SRCQ} is stronger than
  the subregularity assumption, the conclusion of Corollary \ref{SONC2} implies the result of
  \cite[Theorem 2]{Gfrerer19}.
 \end{remark}

 Corollary \ref{SONC2} implies that a no gap second-order
 sufficient optimality condition is
 \begin{equation}\label{SOSC1-ineq}
   \sup_{(\xi,\Gamma)\in \mathcal{M}(\overline{x})}\Big\{\langle d,\nabla^2_{xx}L(\overline{x},\xi,\Gamma)d\rangle
   -\sigma\big((\xi,\Gamma)\,|\,\mathcal{T}^2_{K\times\Omega}(\Upsilon(\overline{x});\Upsilon'(\overline{x})d)\big)\Big\}>0
   \quad\forall d\in\mathcal{C}(\overline{x}).
 \end{equation}
 By Proposition \ref{SO-regular}, under a directional subregularity CQ, we get
 a sufficient condition.
 \begin{theorem}\label{SOSC1}
  Let $\overline{x}$ be a feasible point of the problem \eqref{SDCMPCC} with
  $\mathcal{M}(\overline{x})\ne\emptyset$. Suppose that the inequality
  \eqref{SOSC1-ineq} holds and the multifunction
  $\mathcal{H}(X,Y):=(X;Y;\langle X,Y\rangle)
    -\mathbb{S}_{+}^n\times\mathbb{S}_{-}^n\times\mathbb{R}_{+}$
  is metrically subregular at $\Theta(\overline{x})$ for the origin
  in each $d\in\mathcal{C}(\overline{x})$.
  Then, there exist $\kappa>0$ and $\delta>0$ such that
  for all $z\in\mathcal{S}\cap\mathbb{B}(\overline{x},\delta)$,
  $\varphi(x)\geq \varphi(\overline{x})+\kappa \|x-\overline{x}\|^2$.
  \end{theorem}
  \begin{proof}
  Suppose that the conclusion does not hold at $\overline{x}$.
  Then there exists a sequence of feasible points $\{x^k\}\subseteq\mathcal{S}$
  converging to $\overline{x}$ such that
  \begin{equation}\label{ineq-growth}
   \varphi(x^k)\le \varphi(\overline{x})+o(t_k^2)\ \ {\rm with}\ \
   t_k:=\|x^k-\overline{x}\|.
  \end{equation}
  We can assume, by passing to a subsequence if necessary, that $d^k:=\frac{x^k-\overline{x}}{t_k}$
  converges to a vector $d$ with $\|d\|=1$. Clearly, $d\in\mathcal{C}(\overline{x})$.
  By the Taylor expansion of $\Upsilon(x^k)$ at $\overline{x}$,
  \[
    K\times\Omega\ni\Upsilon(x^k)=\Upsilon(\overline{x})+t_k\Upsilon'(\overline{x})d
    +\frac{1}{2}t_k^2\big[\Upsilon'(\overline{x})w^k+\nabla^2\Upsilon(\overline{x})(d,d)\big]+o(t_k^2)
  \]
  with $w^k:=2t_k^{-2}(x^k-\overline{x}-t_kd)$. Since
  $x^k-\overline{x}-t_kd=o(t_k)$, we have $t_kw^k\to 0$.
  Notice that $\Omega=\mathcal{H}^{-1}(\mathbb{S}_{+}^n\times\mathbb{S}_{-}^n\times\mathbb{R}_{+})$.
  Recall that $\mathbb{S}_{+}^n\times\mathbb{S}_{-}^n\times\mathbb{R}_{+}$ is second-order regular
  at $(\Theta(\overline{x});\langle\theta(\overline{x}),\zeta(\overline{x})\rangle)$
  in the direction $(\Theta'(\overline{x})d;
  \langle\theta'(\overline{x})d,\zeta(\overline{x})\rangle
  +\langle\theta(\overline{x}),\zeta'(\overline{x})d\rangle)$
  by \cite[Proposition 3.136 \& 3.89 \& Example 3.140]{BS00}.
  Together with the subregularity of $\mathcal{H}$ at $\overline{x}$ for
  the origin, from Proposition \ref{SO-regular} it follows that
  $\Omega$ is second-order regular at $\Theta(\overline{x})$ in the direction
  $\Theta'(\overline{x})d$. Together with the second-order regularity of $K$
  and the last equation,
  \[
    \Upsilon'(\overline{x})w^k+\nabla^2\Upsilon(\overline{x})(d,d)\in
    \mathcal{T}_{K\times\Omega}^2(\Upsilon(\overline{x});\Upsilon'(\overline{x})d)
    +o(1)\mathbb{B}_{\mathbb{Y}\times\mathbb{S}^n\times\mathbb{S}^n}.
  \]
  Together with the inequality \eqref{ineq-growth} and the second-order
  Taylor expansion of $\varphi(x^k)$ at $\overline{x}$, there exists
  a sequence $\varepsilon_k\to 0$ such that
  \begin{subequations}
   \begin{align}\label{temp-cond1}
    2t_k^{-1}\varphi'(\overline{x})d+(\varphi'(\overline{x})w^k+\langle d,\nabla^2\varphi(\overline{x})d\rangle)
    \le\varepsilon_k,\qquad\quad\\
    \label{temp-cond2}
    \Upsilon'(\overline{x})w^k+\nabla^2\Upsilon(\overline{x})(d,d)\in
    \mathcal{T}_{K\times\Omega}^2(\Upsilon(\overline{x});\Upsilon'(\overline{x})d)
    +\varepsilon_k\mathbb{B}_{\mathbb{S}^n\times\mathbb{S}^n}.
   \end{align}
  \end{subequations}
  In addition, from \eqref{SOSC1-ineq} there exist $\Gamma\in\mathcal{M}(\overline{x})$
  and a constant $c_1>0$ such that
  \[
    \langle d,\nabla^2_{xx}L(\overline{x},\xi,\Gamma)d\rangle
   -\sigma\big((\xi,\Gamma)\,|\,\mathcal{T}^2_{K\times\Omega}(\Upsilon(\overline{x});\Upsilon'(\overline{x})d)\big)
   \ge c_1.
  \]
  Together with \eqref{temp-cond1}-\eqref{temp-cond2}, it immediately follows that
  \begin{align*}
   0&\ge 2t_k^{-1}\varphi'(\overline{x})d+(\varphi'(\overline{x})w^k+\langle d,\nabla^2\varphi(\overline{x})d\rangle)-\varepsilon_k\\
   &\quad +\langle(\xi,\Gamma),\Upsilon'(\overline{x})w^k\!+\!\nabla^2\Upsilon(\overline{x})(d,d)\rangle
   -\sigma\big((\xi,\Gamma)\,|\,\mathcal{T}^2_{K\times\Omega}(\Upsilon(\overline{x});\Upsilon'(\overline{x})d)\big)
   -\varepsilon_k\|\lambda\|\\
   &=\langle d,\nabla_{xx}^2L(\overline{x},\xi,\Gamma)d\rangle
   -\sigma\big((\xi,\Gamma)\,|\,\mathcal{T}^2_{K\times\Omega}(\Upsilon(\overline{x});\Upsilon'(\overline{x})d)\big)
   -\varepsilon_k(1+\|\lambda\|)\\
   &\ge c_1-\varepsilon_k(1+\|\lambda\|)
  \end{align*}
  where the equality is due to $\varphi'(\overline{x})d=0$ and
  $\langle w^k,\nabla_xL(\overline{x},\xi,\Gamma)\rangle=0$.
  Since $\varepsilon_k\to 0$, the last inequality yields a contradiction.
  Thus, we obtain the desired result.
 \end{proof}

 By Mordukhovich's coderivative rule \cite{Mordu92}, it is not hard to check
 that the multifunction $\mathcal{H}$ in Theorem \ref{SOSC1} is not
 metrically regular at $\Theta(\overline{x})$ for the origin,
 but there is a possibility for it to be metrically subregular
 at this reference point, which is equivalent to a metric qualification
 condition by Lemma \ref{asubregular} in Appendix.
 When the multifunction $\mathcal{H}$ is not
 metrically subregular at $\Theta(\overline{x})$ for the origin,
 by noting that $\mathbb{S}^n_+\times\mathbb{S}_{-}^n$ is second-order regular
 and following the arguments as those for \cite[Theorem 3.83]{BS00},
 we have the following second-order sufficient condition which
 is stronger than the one in Theorem \ref{SOSC1} due to the inclusion relation
 $\mathcal{T}^2_{\mathbb{S}^n_+\times\mathbb{S}_{-}^n}(\Theta(\overline{x});
 \Theta'(\overline{x})d)\supseteq\mathcal{T}_{\Omega}^2(\Theta(\overline{x});\Theta'(\overline{x})d)$.
 \begin{theorem}\label{SOSC2}
  Let $\overline{x}$ be a feasible point of \eqref{ESDCMPCC} with
  $\mathcal{M}(\overline{x})\ne\emptyset$. If for each $d\in\mathcal{C}(\overline{x})$
  \begin{align}\label{SOSC2-ineq}
   \sup_{(\xi,\Gamma)\in \mathcal{M}(\overline{x})}\Big\{\langle d,\nabla^2_{xx}L(\overline{x},\xi,\Gamma)d\rangle
   -\sigma\big((\xi,\Gamma)\,|\, \mathcal{T}^2_{K\times\mathbb{S}^n_+\times\mathbb{S}_{-}^n}(\Upsilon(\overline{x});
   \Upsilon'(\overline{x})d)\big)\Big\}>0,
  \end{align}
  there exist $\kappa>0$ and $\delta>0$ such that
  for all $z\in\mathcal{S}\cap\mathbb{B}(\overline{x},\delta)$,
  $\varphi(x)\geq \varphi(\overline{x})+\kappa \|x-\overline{x}\|^2$.
  \end{theorem}
  \begin{remark}\label{remark-sufficient}
   Let $\overline{x}$ be a feasible point of \eqref{ESDCMPCC}.
   The critical cone of \eqref{ConvexCone} at $\overline{x}$ is
   \begin{align*}
   \widehat{\mathcal{C}}(\overline{x}):=\Big\{d\in \mathbb{X}\ |\ \varphi'(\overline{x})d\leq 0,
   h'(\overline{x})d\in\mathcal{T}_{K}(h(\overline{x})),
   \theta'(\overline{x})d\in \mathcal{T}_{\mathbb{S}^n_+}(\theta(\overline{x})),\\
    \zeta'(\overline{x})d\in \mathcal{T}_{\mathbb{S}^n_-}(\zeta(\overline{x})),
   \langle \theta'(\overline{x})d,\zeta(\overline{x})\rangle
   +\langle \theta(\overline{x}),\zeta'(\overline{x})d\rangle\geq 0\Big\}.
  \end{align*}
  By using \cite[Corollary 3.1 $\&$ Lemma 4.2]{WuZZ14}, it is easy to check that
  $\widehat{\mathcal{C}}(\overline{x})\subset\mathcal{C}(\overline{x})$.
  This, along with \cite[Lemma 4.2]{WuZZ14}, means that the second-order
  sufficient condition \eqref{SOSC2-ineq} and so \eqref{SOSC1-ineq}
  is implied by the SOSC in \cite[Definition 4.1]{WuZZ14}.
  When the strict complementarity holds at $\overline{x}$,
  the sufficient condition \eqref{SOSC2-ineq} coincides with
  the one in \cite[Definition 4.2]{WuZZ14}.
 \end{remark}
 \section{Application to rank optimization problems}\label{sec6}

  Let $\vartheta\!:\mathbb{S}^{n}\to\mathbb{R}_{+}$ be a twice differentiable loss.
  Consider the rank-regularized problem:
  \begin{equation}\label{rank-reg}
   \min_{X\in\mathbb{S}^n}\big\{\vartheta(X)+{\rm rank}(X)\big\}.
  \end{equation}
  From \cite[Section 4.2]{LiuBiPan18}, this problem can be reformulated as
  the following SDCMPCC
  \begin{equation}\label{MPEC}
   \min_{X\in\mathbb{S}^n,W\in \mathbb{S}_{+}^n}\Big\{\vartheta(X)+{\rm tr}(W)
    \ \ {\rm s.t.}\ \ \langle X,W\!-I\rangle=0,\,X\in\mathbb{S}_{+}^n,\,W\!-I\in\mathbb{S}_{-}^n\Big\},
  \end{equation}
  which has the form of \eqref{SDCMPCC} with
  $K=\mathbb{S}_{+}^n,\varphi(X,W):=\vartheta(X)+{\rm tr}(W)$,
  $h(X,W):=W$ and $\Theta(X,W):=(X;W-I)$ for $(X,W)\in\mathbb{S}^n\times\mathbb{S}^n$.
  Let $X^*$ be a local optimal solution of \eqref{rank-reg} with
  the eigenvalue decomposition as $X^*=P^*{\rm Diag}(\lambda(X^*)){P^*}^{\mathbb{T}}$.
  Write $r={\rm rank}(X^*)$ and take $W^*=P_1^*(P_1^*)^{\mathbb{T}}$ where
  $P_1^*$ is the submatrix consisting of the first $r$ columns of $P^*$.
  It is easy to check that $(X^*,W^*)$ is a local optimal solution to \eqref{MPEC}.

  \medskip

  We first argue that the condition in \eqref{SRCQ} holds at $(X^*,W^*)$.
  For this purpose, let $Z=X^*+(W^*\!-I)$ have the spectral decomposition
  as in \eqref{Zdecom}-\eqref{idx-eigZ}. Then,
  \begin{equation}\label{equa1-XYstar}
    X^*=P\left[\begin{matrix}
              \Lambda_{\alpha} & 0 & 0\\
              0 & 0_{\beta\beta} & 0\\
              0 & 0 & 0_{\gamma\gamma}
            \end{matrix}\right]P^{\mathbb{T}}
    \ \ {\rm and}\ \
    W^*-I=P\left[\begin{matrix}
              0_{\alpha\alpha} & 0 & 0\\
              0 & 0_{\beta\beta} & 0\\
              0 & 0 & \Lambda_{\gamma}
            \end{matrix}\right]P^{\mathbb{T}}
  \end{equation}
  with $\Lambda_{\alpha}\!={\rm Diag}(z^*)$ for $z^*\in\!\mathbb{R}_{+}^{|\alpha|}$
  and $\Lambda_{\gamma}\!={\rm Diag}(w^*)$ for $w^*\in[-e,0)$.
  Take an arbitrary $(H,F,G)\in{\rm Sp}\big(\mathcal{N}_{\mathbb{S}_{+}^n\times\Omega}^{c}(\Upsilon(X^*,W^*))\big)
  \cap{\rm Ker}(\nabla\Upsilon(X^*,W^*))$. Then, $F=0$ and $G+H=0$.
  Notice that ${\rm Sp}(\mathcal{N}_{\mathbb{S}_{+}^n\times\Omega}^{c}(\Upsilon(X^*,W^*)))
  \subseteq[{\rm lin}(\mathcal{T}_{\mathbb{S}_{+}^n}(W^*))]^{\circ}\times
  [\mathcal{T}_{\Omega}^{c}(X^*,W^*\!-I)]^{\circ}$ and
  \[
    [{\rm lin}(\mathcal{T}_{\mathbb{S}_{+}^n}(W^*))]^{\circ}
    =\big\{H\in\mathbb{S}^n\ |\ P_{\gamma'}^{\mathbb{T}}HP_{\gamma'}=0\big\}^{\circ}
    \ \ {\rm with}\ \gamma'=\{i\in\gamma\ |\ w_i^*=-1\}.
  \]
  Together with $(H,F,G)\in[{\rm lin}(\mathcal{T}_{\mathbb{S}_{+}^n}(W^*))]^{\circ}\times
  [\mathcal{T}_{\Omega}^{c}(X^*,W^*\!-I)]^{\circ}$
  and Lemma \ref{atangent}, we have
  \[
     P^{\mathbb{T}}HP=
    \left[\begin{matrix}
          0 & 0 & 0\\
          0  & 0 & 0\\
          0 & 0 & \widetilde{H}_{\gamma'\gamma'}
          \end{matrix}\right]
    {\ \ \rm and \ \ }
    P^{\mathbb{T}}G P=
    \left[\begin{matrix}
         \widetilde{G}_{\alpha\alpha} & \widetilde{G}_{\alpha\beta} & 0_{\alpha\gamma}\\
         \widetilde{G}_{\alpha\beta}^\mathbb{T} & 0_{\beta\beta}& 0\\
          0_{\gamma\alpha} & 0 & 0
          \end{matrix}\right].
  \]
  Combining this with $G+H=0$, we derive $H=0$ and $G=0$.
  Thus, $(H,F,G)=(0,0,0)$. So, the condition \eqref{SRCQ} holds.
  From Corollary \ref{SONC2}, the following result holds for \eqref{rank-reg}.
  \begin{proposition}\label{RankTh1}
   Let $(X^*,W^*)$ be a local optimal solution of \eqref{MPEC}.
   Then, $\mathcal{M}(X^*,W^*)$ is a singleton,
   say ${(S^*,\Gamma^*)}$, and for any $d=(G,H)$ satisfying
   $\langle\nabla\vartheta(X^*),G\rangle+{\rm tr}(H)\le 0$,
   $(G,H)\in \mathcal{T}_{\Omega}(X^*,W^*\!-\!I)$ and
   $H\in \mathcal{T}_{\mathbb{S}^n_+}(W^*)$,
   it holds that
   \[
   \langle d,\nabla^2L((X^*,W^*),S^*,\Gamma^*)d\rangle
    -\sigma\big(S^*\,|\,\mathcal{T}^2_{\mathbb{S}^n_+}(W^*;H))
    -\sigma\big(\Gamma^*\,|\,
    \mathcal{T}^2_{\Omega}((X^*,W^*\!-I);(G,H))\big)\geq 0.
   \]
  \end{proposition}

  To close this section, we illustrate the role of the outer second-order tangent set
  in the second-order necessary condition by the following example.
 \begin{example}\label{example}
  Consider the problem \eqref{MPEC} with
  $\vartheta(X)=\frac{1}{2}(X_{11}\!-1)^2+X_{33}+2X^2_{12}$ and $n=3$.
  Obviously, $(X^*,W^*)$ with $X^*={\rm Diag}(1,0,0)$ and $W^*={\rm Diag}(1,0,0)$
  is an optimal solution. By \cite[Proposition 3.2]{DingSY14},
  we calculate that $\mathcal{M}(X^*,W^*)={(\xi,(\Gamma^1,\Gamma^2))}$ with
  \begin{equation*}\label{Gamma1-2}
    \xi=\left(\begin{matrix}
                        0 & 0 &0\\
                       0 & -1 & 0\\
                       0 & 0 &-1
                       \end{matrix}\right), \ \
    \Gamma^1=\left(\begin{matrix}
                        0 & 0 &0\\
                       0 & 0 & 0\\
                       0 & 0 &-1
                       \end{matrix}\right)
    {\rm \ \ and \ \ }
     \Gamma^2=\left(\begin{matrix}
                        -1 & 0 & 0\\
                        0 & 0 & 0\\
                        0 & 0 & 0
                       \end{matrix}\right).
   \end{equation*}
   Write $Y^*:=W^*-I$. From \cite[Theorem 3.1]{WuZZ14}, it follows that
   \begin{align*}
   &\mathcal{C}(X^*,W^*)=\{(G,H)\,|\, (H;(G,H))\in \mathcal{T}_{\mathbb{S}^3_+\times\Omega}(W^*;(X^*,Y^*)),\varphi'(X^*,W^*)(G,H)= 0\} \\
   &=\Big\{(G,H)\,|\, G=\left(\begin{matrix}
                        G_{11} & G_{12} & G_{13}\\
                        G_{12} &0 & 0\\
                        G_{13} & 0 & 0
                       \end{matrix}\right), H=\left(\begin{matrix}
                        0 & G_{12} & G_{13} \\
                       G_{12}& 0 & 0 \\
                       G_{13} & 0 & 0
                       \end{matrix}\right),G_{11},G_{12},G_{13}\in \mathbb{R} \Big\}.
   \end{align*}
   Now take an arbitrary direction $d=(G,H)\in \mathcal{C}(X^*,W^*)$.
   From Theorem \ref{SO-tangent}, we know that $(S,T)\in \mathcal{T}^2_{\Omega}((X^*,Y^*);G+H)$
   if and only if $S$ and $T$ take the following form
   \[
      S=\left[\begin{matrix}
                S_{11} &  T_{12}+2G_{11}G_{12} & T_{13}+2G_{11}G_{13}\\
                T_{12}+2G_{11}G_{12} & 2G^2_{12} & 2G_{12}G_{13} \\
                T_{13}+2G_{11}G_{13} &  2G_{12}G_{13} & 2G^2_{13}
       \end{matrix}\right]
   \]
   and
  \[
  T=\left[\begin{matrix}
                -2(G^2_{12}+G^2_{13}) & T_{12} & T_{13}\\
                T_{12} & T_{22} & T_{23}\\
                T_{13} & T_{23} &T_{33}

       \end{matrix}\right].
  \]
  Together with the expressions of $\Gamma^1$ and $\Gamma^2$,
  it is not difficult to obtain
  \[
   \sigma\big((\xi,\Gamma^1,\Gamma^2)\,|\,\mathcal{T}^{2}_{\mathbb{S}^n_+\times \Omega}((W^*,X^*,Y^*);(H,G,H))\big)=2+2G^2_{13}.
  \]
  By comparing $\langle d,\nabla^2_{xx}L((X^*,W^*),\xi,(\Gamma^1,\Gamma^2))d\rangle
   =G^2_{11}+8G^2_{12}$ with the following equality
  \begin{align*}
   &\nabla^2_{xx}L((X^*,W^*),\xi,(\Gamma^1,\Gamma^2))(d,d)
   -\sigma((\xi,\Gamma^1,\Gamma^2)\,|\,{\mathcal{T}}^2_{\Omega}((W^*,X^*,Y^*);(H,G,
   H))\nonumber\\
   &=G^2_{11}+8G^2_{12}-2-2G^2_{13},
  \end{align*}
  we conclude that the second-order necessary conditions involving the second-order tangent set
  is stronger than the one not involving second-order tangent set.
 \end{example}


  \bigskip
  \noindent
  {\bf\large Appendix}

 \begin{alemma}\label{Maincor}
  Fix an arbitrary $(X,Y)\in \Omega$ and an arbitrary $(F,G)\in \mathcal{T}_{\Omega}(X,Y)$.
  Let $Z=X+Y$ have the spectral decomposition as in \eqref{Zdecom}-\eqref{idx-eigZ}.
  Then,
  \begin{itemize}
     \item [(i)] when $X\in{\rm int}(\mathbb{S}_{+}^n)$ or $X\in\mathbb{S}_{+}^n$ with
                 $Y\!=G\!=0$ and $F\!\in \!{\rm ri}(\mathcal{T}_{\mathbb{S}^n_+}(X))$, it holds that
                 \[
                   \mathcal{T}^{2}_{\Omega}((X,Y);(F,G))=\mathbb{S}^n\times \{0\};
                 \]

     \item [(ii)] when $Y\in{\rm int}(\mathbb{S}_{-}^n)$ or $Y\in\mathbb{S}_{-}^n$ with
                  $X\!=F\!=0$ and $G\!\in \!{\rm ri}(\mathcal{T}_{\mathbb{S}^n_{-}}(Y))$, it holds that
                  \[
                    \mathcal{T}^{2}_{\Omega}((X,Y);(F,G))=\{0\}\times\mathbb{S}^n;
                  \]

     \item [(iii)] when $X=Y=0$, $\mathcal{T}^{2}_{\Omega}((X,Y);(F,G))=\mathcal{T}_{\Omega}(F,G)$;

     \item [(iv)] when $X=F=0,Y\in\!{\rm bd}(\mathbb{S}^n_{-})\backslash \{0\},
                  G\in\!{\rm rb}(\mathcal{T}_{\mathbb{S}^n_-}(Y))$,
                  $(S,T)\in \mathcal{T}^{2}_{\Omega}((X,Y);(F,G))$ iff
     \begin{subnumcases}{\label{CoroRe1}}
      \label{CoroRe11}S= P\left[\begin{matrix}
                \widetilde{S}_{a_1a_1}& 0\\
                0 & 0
       \end{matrix}\right]P^{\mathbb{T}}
      {\ \ \rm and \ \ } T=P\widetilde{T} P^{\mathbb{T}}\\
     \label{CoroRe12}\widetilde{S}_{a_1a_1}=Q^1\left[\begin{matrix}
                (\widehat{S}_{a_1a_1})_{\beta^1\beta^1}& 0\\
                0 & 0
       \end{matrix}\right](Q^1)^{\mathbb{T}} \\
     \widetilde{T}_{a_1a_1}=Q^1\left[\begin{matrix}
         (\widehat{T}_{a_1a_1})_{\beta^1\beta^1}&(\widehat{T}_{a_1a_1})_{\beta^1\gamma^1}\\
         \Big((\widehat{T}_{a_1a_1})_{\beta^1\gamma^1}\Big)^{\mathbb{T}} &(\widehat{T}_{a_1a_1})_{\gamma^1\gamma^1}
       \end{matrix}\right](Q^1)^{\mathbb{T}}\\
      \label{CoroRe13}(\widehat{S}_{a_1a_1})_{\beta^1\beta^1}=\Pi_{\mathbb{S}_{+}^{|\beta^1|}}
     \Big((\widehat{S}_{a_1a_1}+\widehat{T}_{a_1a_1}
     -\Delta^-)_{\beta^1\beta^1}\Big)
     \end{subnumcases}

    \item [(v)] when $X\in\!{\rm bd}(\mathbb{S}^n_{+})\backslash\{0\},Y=G=0,
                F\in\!{\rm rb}(\mathcal{T}_{\mathbb{S}^n_+}(Y))$,
               $(S,T)\in \mathcal{T}^{2}_{\Omega}((X,Y);(F,G))$ iff
      \begin{subnumcases}{\label{CoroRe2}}
      \label{CoroRe21}S=P\widetilde{S} P^{\mathbb{T}}
      {\ \ \rm and \ \ }
      T= P\left[\begin{matrix}
                0& 0\\
                0 & \widetilde{T}_{a_ra_r}
       \end{matrix}\right]P^{\mathbb{T}}\\
     \label{CoroRe22}\widetilde{S}_{a_ra_r}=Q^r\left[\begin{matrix}
                (\widehat{S}_{a_ra_r})_{\alpha^r\alpha^r}& (\widehat{S}_{a_ra_r})_{\alpha^r\beta^r}\\
                \Big((\widehat{S}_{a_ra_r})_{\alpha^r\beta^r}\Big)^{\mathbb{T}} & (\widehat{S}_{a_ra_r})_{\beta^r\beta^r}
       \end{matrix}\right](Q^r)^{\mathbb{T}} \\
     \widetilde{T}_{a_ra_r}=Q^r\left[\begin{matrix}
         0&0\\
         0 &(\widehat{T}_{a_r a_r})_{\beta^r\beta^r}
       \end{matrix}\right](Q^r)^{\mathbb{T}}\\
     \label{CoroRe23}(\widehat{S}_{a_ra_r}-\Delta^+)_{\beta^r\beta^r}=\Pi_{\mathbb{S}_{+}^{|\beta^r|}}
     \Big((\widehat{S}_{a_ra_r}+\widehat{T}_{a_ra_r}
     -\Delta^+)_{\beta^r\beta^r}\Big)
     \end{subnumcases}
    \end{itemize}
  \end{alemma}
  \begin{proof}
  The first three cases follows by the definition of second-order tangent set.

  \medskip
  \noindent
  {\bf(iv)} Since $X=F=0, Y\in{\rm bd}(\mathbb{S}^n_{-})\backslash\{0\}$
  and $G\in {\rm ri}(\mathcal{T}_{\mathbb{S}^n_-}(Y))$, it follows that
  $Z=Y,\alpha=\emptyset$, $\mu_1=0$ and $\mu_i<0$ for each $i\in\{2,\ldots,r\}$
  and $\Delta^+=0$ in \eqref{TKKResult2}.
  Notice that $H=F+G\in {\rm ri}(\mathcal{T}_{\mathbb{S}^n_-}(Y))$ and
  ${\rm ri}(\mathcal{T}_{\mathbb{S}^n_-}(Y))
   =\{ \Gamma\in \mathbb{S}^n\,|\, P^\mathbb{T}_{a_1}\Gamma P_{a_1}\prec 0\}$.
  We have $\eta^1_j<0$ for each $j=1,2,\ldots, N_1$ and $\alpha^1=\beta^1=\emptyset$. By \eqref{TKKResult2}, $\widetilde{S}_{a_1a_1}=0$.
  From \eqref{TKKResult3}, $\widetilde{S}_{a_ka_k}=0$ for each $k\in \{2,\ldots,r\}$. From \eqref{TKLResult4}, $\widetilde{S}_{a_1a_l}=0$ for any $l\in \{2,\ldots,r\}$ and
  from \eqref {TKLResult5}, $\widetilde{S}_{a_ka_l}=0$ for any $k<l$ and $k,l\in \{2,\ldots,r\}$.
  So, $\widetilde{S}=0$ and $S=0$. The result follows.

  \medskip
  \noindent
  {\bf(v)} Since $X\in {\rm bd}\mathbb{S}^n_{+}\setminus \{0\},Y=G=0$ and
  $F\in {\rm ri}(\mathcal{T}_{\mathbb{S}^n_+}(Y))$, we know that $Z=X$
  and $\gamma=\emptyset$ and $\alpha,\beta$ are nonempty and $\mu_r=0$
  and $\mu_i>0$ for each $i\in\{1,2,\ldots,r-1\}$ and $\Delta^-=0$
  in \eqref{TKKResult2}. Notice that $H=F+G\in {\rm ri}(\mathcal{T}_{\mathbb{S}^n_+}(X))$
  and ${\rm ri}(\mathcal{T}_{\mathbb{S}^n_+}(X))
  =\{\Gamma\in \mathbb{S}^n\,|\, P^\mathbb{T}_{a_r}\Gamma P_{a_r}\succ 0\}$.
  We have $\eta^r_j>0$ for each $j=1,2,\cdots, N_r$ and
  then $\beta^r=\gamma^r=\emptyset$. By \eqref{TKKResult1} and \eqref{TKKResult2},
  $\widetilde{T}_{a_ka_k}=0$ for each $k\in\{1,2,\ldots,r\}$.
  From \eqref{TKLResult1}, $\widetilde{T}_{a_ka_l}=0$ for each $k,l\in\{1,2,\ldots,r\!-1\}$
  with $k<l$ and from \eqref {TKLResult2}, $\widetilde{T}_{a_ka_r}=0$ for each $k\in\{1,2,\ldots,r\!-1\}$.
  Hence, $\widetilde{T}=0$ and then $T=P\widetilde{T} P^{\mathbb{T}}=0$. Consequently, $\mathcal{T}^{2}_{\Omega}((X,Y);(F,G))=\mathbb{S}^n\times \{0\}$.
 \end{proof}
 \begin{alemma}\label{atangent}
  Fix an arbitrary $(X,Y)\in\Omega$ and let $Z=X+Y$ have the eigenvalue
  decomposition as in \eqref{Zdecom}-\eqref{idx-eigZ}.
  Then, $(F,G)\in\mathcal{T}^c_{\Omega}(X,Y)$ if and only if $(F,G)$ satisfies
  \begin{subnumcases}{}\label{GHmatrix3}
    P^{\mathbb{T}}FP=\left[\begin{matrix}
         \widetilde{F}_{\alpha\alpha} & \widetilde{F}_{\alpha\beta}&\widetilde{F}_{\alpha\gamma}\\
        \widetilde{F}_{\alpha\beta}^{\mathbb{T}}&0& 0\\
        \widetilde{F}_{\alpha\gamma}^{\mathbb{T}}& 0 & 0
       \end{matrix}\right],\ \
    P^{\mathbb{T}}GP=\left[\begin{matrix}
        0 & 0&\widetilde{G}_{\alpha\gamma}\\
        0 &0&\widetilde{G}_{\beta\gamma}\\
        \widetilde{G}_{\alpha\gamma}^{\mathbb{T}}& \widetilde{G}_{\beta\gamma}^{\mathbb{T}} &\widetilde{G}_{\gamma\gamma}
       \end{matrix}\right],\\
   \label{GHmatrix4}
   (\Sigma_{\alpha\gamma}-E_{\alpha\gamma})\circ\widetilde{F}_{\alpha\gamma}
   +\Sigma_{\alpha\gamma}\circ\widetilde{G}_{\alpha\gamma}=0,
  \end{subnumcases}
  which implies that
  $\mathcal{T}_{\Omega}(X,Y)+\mathcal{T}_{\Omega}^{c}(X,Y)=\mathcal{T}_{\Omega}(X,Y)$
  and $(\Delta F,\Delta G)\in[\mathcal{T}_{\Omega}^{c}(X,Y)]^{\circ}$ iff
  \begin{subnumcases}{}\label{GHmatrix3}
    P^{\mathbb{T}}\Delta FP=\left[\begin{matrix}
        0 & 0&\Delta\widetilde{F}_{\alpha\gamma}\\
        0 &0&\Delta\widetilde{F}_{\beta\gamma}\\
        \Delta\widetilde{F}_{\alpha\gamma}^{\mathbb{T}}& \Delta\widetilde{F}_{\beta\gamma}^{\mathbb{T}} &\Delta\widetilde{F}_{\gamma\gamma}
       \end{matrix}\right],
   P^{\mathbb{T}}\Delta GP=\left[\begin{matrix}
        \Delta\widetilde{G}_{\alpha\alpha} & \Delta\widetilde{G}_{\alpha\beta}&\Delta\widetilde{G}_{\alpha\gamma}\\
        \Delta\widetilde{G}_{\alpha\beta}^{\mathbb{T}}&0& 0\\
        \Delta\widetilde{G}_{\alpha\gamma}^{\mathbb{T}}& 0 & 0
       \end{matrix}\right],\\
   \label{GHmatrix4}
   \Sigma_{\alpha\gamma}\circ\Delta\widetilde{F}_{\alpha\gamma}
   +(E_{\alpha\gamma}-\Sigma_{\alpha\gamma})\circ\Delta\widetilde{G}_{\alpha\gamma}=0.
  \end{subnumcases}
 \end{alemma}
 \begin{proof}
  By the characterization of $\mathcal{N}_{{\rm gph}\,\mathcal{N}_{\mathbb{S}_{+}^{|\beta|}}}(0,0)$
  in \cite[Proposition 3.3]{DingSY14}, it is not hard to check that
  $[\mathcal{N}_{{\rm gph}\,\mathcal{N}_{\mathbb{S}_{+}^{|\beta|}}}(0,0)]^{\circ}=\{(0,0)\}$.
  Together with \cite[Theorem 3.1]{DingSY14}, it follows that
  $(F,G)\in\mathcal{T}^c_{\Omega}(X,Y)=[\mathcal{N}_{\Omega}(X,Y)]^{\circ}$
  iff $(F,G)$ satisfies \eqref{GHmatrix3}-\eqref{GHmatrix4}.
  Combining this fact with the characterization of $\mathcal{T}_{\Omega}(X,Y)$
  in \cite[Corollary 3.1]{WuZZ14}, we get the second part.
 \end{proof}
 \begin{alemma}\label{asubregular}
  The subregularity of $\mathcal{H}$ in Theorem \ref{SOSC1} at $\Theta(x^*)$
  for the origin is equivalent to the existence of $\varepsilon>0$ and $\kappa>0$
  such that for all $(X,Y)\in\mathbb{B}(\Theta(x^*),\delta)$,
  \begin{equation}\label{metric-CQ}
    {\rm dist}((X,Y),\Omega)
    \le\kappa\big[{\rm dist}((X,Y),S)+{\rm dist}((X,Y),\mathbb{S}_{+}^n\times\mathbb{S}_{-}^n)\big]
  \end{equation}
  where $S:=\{(X,Y)\in\mathbb{S}^n\times\mathbb{S}^n\ |\ \langle X,Y\rangle=0\}$.
 \end{alemma}
 \begin{proof}
 Define the multifunctions
 $\mathcal{G}\!:\mathbb{S}^n\times\mathbb{S}^n\to\mathbb{S}^n\times\mathbb{S}^n$
 and $\mathcal{R}\!:\mathbb{S}^n\times\mathbb{S}^n\to\mathbb{R}$ by
 \begin{subequations}
  \begin{align}
   \mathcal{G}(U,V)&:=\left\{(X,Y)\in\mathbb{S}^n\times\mathbb{S}^n\ |\
   \left(\begin{matrix}
       X-U\\
       Y-V
   \end{matrix}\right)\in
   \left(\begin{matrix}
       \mathbb{S}_{+}^n\\
       \mathbb{S}_{-}^n
   \end{matrix}\right)\right\},\\
   \mathcal{R}(W)&:=\left\{(X,Y)\in\mathbb{S}^n\times\mathbb{S}^n\ |\ \langle X,Y\rangle-W\in\mathbb{R}_{+}\right\}.
   \end{align}
 \end{subequations}
 Clearly, $\mathcal{H}^{-1}(U,V,W)=\mathcal{G}(U,V)\cap\mathcal{R}(W)$.
 Write $(X^*,Y^*)=\Theta(x^*)\in\Omega$. Clearly,
 the multifunction $\mathcal{G}$ is calm at $(0,0,X^*,Y^*)$.
 In addition, by invoking Mordukhovich's coderivative rule \cite{Mordu92},
 we can check that $\mathcal{G}^{-1}$ has the Aubin property
 at $(X^*,Y^*)$ for the origin. Now from \cite[Theorem 3.6]{Kummer02}
 it follows that the multifunction $\mathcal{H}^{-1}$ is calm
 at the origin for $(X^*,Y^*)$, or equivalently, the multifunction
 $\mathcal{H}$ is subregular at $(X^*,Y^*)$ for the origin,
 whenever $(U,V)\rightrightarrows \mathcal{R}(0)\cap\mathcal{G}(U,V)$
 is calm at the origin for $(X^*,Y^*)$, i.e.,
 \[
   \mathcal{E}(X,Y):=\left\{(X,Y)\in S\ |\
   \left(\begin{matrix}
       X-U\\
       Y-V
   \end{matrix}\right)\in
   \left(\begin{matrix}
       \mathbb{S}_{+}^n\\
       \mathbb{S}_{-}^n
   \end{matrix}\right)\right\}
 \]
 is subregular at $(X^*,Y^*)$ for the origin. This, by the definition of
 the subregularity, is equivalent to the condition in \eqref{metric-CQ}.
 The proof is completed.
 \end{proof}


\begin{thebibliography}{1}

  \bibitem{BiPan17}
  {\sc S. J.\ Bi and S. H.\ Pan},
  {\em Multi-stage convex relaxation approach to rank regularized minimization problems
  based on equivalent MPGCCs}, SIAM Journal on Control and Optimization, 55(2017): 2493-2518.

   \bibitem{BonnansCS99}
   {\sc J. F.\ Bonnans, R.\ Cominetti and A.\ Shapiro},
   {\em Second-order optimality conditions based on parabolic second-order tangent sets},
   SIAM Journal on Optimization, 9(1999): 466-492.

   \bibitem{BonnansR05}
   {\sc J. F.\ Bonnans,  and H.\ Ram\'{\i}rez C.},
   {\em Perturbation analysis of second-order cone programming problems},
   Mathematical Programming, 104(2005): 205-227.

  \bibitem{BS00}
  {\sc J. F.\ Bonnans, and A.\ Shapiro,}
  {\em Perturbation Analysis of Optimization Problems},
  Springer, New York. 2000.


  \bibitem{ChenYe19}
  {\sc J. S.\ Chen, J. J.\ Ye, J.\ Zhang and J. C.\ Zhou},
  {\em Exact formula for the second-order tangent set of the second-order cone complementarity set},
  arXiv:1906.099/6v1, 2019.


  \bibitem{Constantin06}
  {\sc E.\ Constantin},
  {\em Second-order necessary conditions based on second-order tangent cones},
   Mathematical Sciences Research Journal, 10(2006): 42-56.





  \bibitem{DingSY14}
  {\sc C.\ Ding, D. F.\ Sun and J. J.\ Ye},
  {\em First order optimality conditions for mathematical programs with semidefinite cone complementarity},
  Mathematical Programming, Ser. A, 147(2014): 539-579.



%


%

  \bibitem{Gfrerer11}
  {\sc H.\ Gfrerer},
  {\em First-order and second-order characterizations of metric subregularity and
  calmness of constraint set mappings},
  SIAM Journal on Optimization, 21(2011): 1439-1474.

  \bibitem{Gfrerer13}
  {\sc H.\ Gfrerer},
  {\em On directional metric regularity, subregularity and optimality conditions for nonsmooth mathematical programs},
  Set-Valued and Variational Analysis, 21(2013): 151-176.
  
  \bibitem{Gfrerer14}
  {\sc H.\ Gfrerer},
  {\em Optimality conditions for disjunctive programs based on generalized differentiation with application to mathematical programs with equilibrium constraints},
  SIAM Journal on Optimization, 24(2014): 898-931.


  \bibitem{Gfrerer16}
  {\sc H.\ Gfrerer and J. V.\ Outrata},
  {\em On computation of generalized derivative of the normal-cone mapping and their applications},
  Mathematics of Operations Research, 41(2016): 1535-1556.

  \bibitem{Gfrerer162}
  {\sc H.\ Gfrerer and J. V.\ Outrata},
  {\em On Lipschitzian properties of implicit multifunctions},
  SIAM Journal on Optimization, 26(2016): 2160-2189.

  \bibitem{Gfrerer19}
  {\sc H.\ Gfrerer, J. J.\ Ye and J. C.\ Zhou},
  {\em Second-order optimality conditions for non-convex set-constrained optimization problems},
  arXiv:1911.04076v1, 2019.


  \bibitem{Ginchev11}
  {\sc I.\ Ginchev  and B. S.\ Mordukhovich},
  {\em On directionally dependent subdifferentials},
  Comptes rendus de I'Acad\'{e}mie bulgare des sciences: sciences math\'{e}matiques et naturelles, 64(2011): 497-508.


  \bibitem{LuoYe13}
  {\sc L.\ Guo, G. H. \ Lin and J. J. \ Ye},
  {\em Second-order optimality conditions for mathematical programs with equilibrium constraint},
  Jounal of Optimzation Theory and Applications, 158(2013): 33-64.



  \bibitem{Henrion05}
  {\sc H.\ Henrion and J. V.\ Outrata},
  {\em Calmness of constraint systems with applications},
   Mathematical programming, 104 (2005): 437-464.

  \bibitem{Ioffe08}
  {\sc A. D.\ Ioffe and J. V.\ Outrata},
  {\em On metric and calmness qualification conditions in subdifferential calculus},
  Set-Valued and Variational Analysis, 16(2008): 199-227.
  
  \bibitem{Kummer02}
  {\sc D.\ Klatte and B. \ Kummer},
  {\em Contrained minima and lipschitzian penalties in metric spaces},
  SIAM Journal on Optimization, 13(2002): 619-633.
  
  \bibitem{Liang14}
  {\sc Y. C.\ Liang, X. D. Zhu and G. H.\ Lin},
  {\em Necessary optimality conditions for mathematical programs with second-order cone complementarity constraints},
  Set-Valued and Variational Analysis, 22(2014): 59-78.




  \bibitem{LiuBiPan18}
  { Y. L. \ Liu, S. J. \ Bi and  S. H.\  Pan},
  {\em Equivalent Lipschitz surrogates for zero-norm and rank optimization problems},
  Journal of Global Optimization, 72(2018): 679-704.
  
  \bibitem{Luo96}
  {\sc Z. Q.\ Luo, J. S. \ Pang and D. Ralph},
  {\em Mathematical Programs with Equilibrium Constraints}, Cambridge University Press, 1996.

  \bibitem{Mohammadi19}
  {\sc A. \ Mohammadi, B. S.\ Mordukhovich and M. E. \ Sarabi},
  {\em Parabolic regularity in geometric variational analysis},
  arXiv:1909.00241v1.


  \bibitem{Mordu92}
  {\sc B. S.\ Mordukhovich,}
  {\em Sensitivity analysis in nonsmooth optimization},
  Theoretical Aspects of Industrial Design(D. A. Field and V. Kmokov,eds.) SIAM Philadelphia, 58(1992): 32-46.

  \bibitem{Mordu06}
  {\sc B. S.\ Mordukhovich,}
  {\em Variational Analysis and Generalized Differentiation I},
  Springer, New York. 2006.



%
%
%
%

  \bibitem{Roc70}
  {\sc R. T.\ Rockafellar},
  {\em Convex Analysis}, Princeton University Press, 1970.

 \bibitem{RW98}
 {\sc R. T.\ Rockafellar and R. J-B.\ Wets},
 {\em Variational Analysis}, Springer, 1998.



  \bibitem{Scheel00}
  {\sc H. S.\ Scheel and S. Scholtes},
  {\em Mathematical programs with complementarity constraints: stationarity, optimality, and sensitivity},
  Mathematics of Operations Research, 25(2000): 1-22.

%

   \bibitem{WuZZ14}
  {\sc J.\ Wu, L. W.\ Zhang and Y.\ Zhang},
  {\em Mathematical programs with semidefinite cone complementarity constraints: constraint qualifications and optimality conditions}, Set-Valued and Variational Analysis,  22(2014): 155-187.
  
  \bibitem{YeZhou16}
 {\sc J. J. \ Ye and J. C. \ Zhou},
 {\em  First order optimality conditions for mathematical programs with second-order cone complementarity constraints},
 SIAM Journal on Optimization,  26(2016): 2820-2846.
 
 
  \bibitem{YeZhou18}
 {\sc J. J. \ Ye and J. C. \ Zhou},
 {\em  Verifiable sufficient conditions for error bound property of second-order cone complementarity problems},
 Mathematical Programming,  171(2018): 361-395.


  \bibitem{ZhangZX13}
  {\sc  L. W.\ Zhang, N.\ Zhang and X. T.\ Xiao},
  {\em On the second-order directional derivatives of singular values of matrices and symmetric matrix-valued functions}, Set-Valued and Variational Analysis, 21(2013): 557-586.
  
  \bibitem{ZhangWZ15}
  {\sc  Y.\ Zhang, J. \ Wu and L. W. \ Zhang},
  {\em First order necessary optimality conditions for mathematical programs with second-order cone complementarity constraints}, Jounral of Global Optimization, 63(2015): 253-279.


 

 
\end{thebibliography}
 \end{document}